\documentclass[11pt]{amsart}
\usepackage[latin9]{inputenc}
\usepackage{geometry}
\usepackage{amstext}
\usepackage{amsthm}
\usepackage{amssymb}

\makeatletter
\numberwithin{equation}{section}
\numberwithin{figure}{section}

\usepackage{amscd}
\usepackage{amsfonts}
\usepackage{mathrsfs}
\usepackage{euscript}
\usepackage{xcolor}
\usepackage[alphabetic]{amsrefs}

\theoremstyle{plain}
\newtheorem{theorem}{Theorem}[section]
\newtheorem{proposition}[theorem]{Proposition}
\newtheorem{lemma}[theorem]{Lemma}
\newtheorem{corollary}[theorem]{Corollary}

\theoremstyle{definition}
\newtheorem{definition}[theorem]{Definition}
\newtheorem{example}[theorem]{Example}

\theoremstyle{remark}
\newtheorem{remark}[theorem]{Remark}

\global\long\def\Ric{\mathrm{Ric}}
\global\long\def\Hess{\mathrm{\mathrm{Hess}}}
\global\long\def\scal{\mathrm{\mathrm{scal}}}
\global\long\def\bbR{\mathbb{\mathbb{R}}}

\global\long\def\div{\mathrm{div}}

\global\long\def\tr{\mathrm{tr}}

\DeclareMathOperator{\ric}{Ric}
\DeclareMathOperator{\hess}{Hess}

\global\long\def\bbR{\mathbb{\mathbb{R}}}

\makeatother

\begin{document}
\title{Rigidity of Homogeneous Gradient Soliton Metrics and Related Equations }
\author{Peter Petersen}
\address{520 Portola Plaza\\
 Dept. of Mathematics, UCLA\\
 Los Angeles, CA 90095.}
\email{petersen@math.ucla.edu}
\urladdr{https://www.math.ucla.edu/\textasciitilde petersen/}
\author{William Wylie}
\address{215 Carnegie Building\\
 Dept. of Math, Syracuse University\\
 Syracuse, NY, 13244.}
\email{wwylie@syr.edu}
\urladdr{https://wwylie.expressions.syr.edu}
\thanks{The second author was supported by grant NSF-\#1654034.}
\begin{abstract}
We prove structure results for homogeneous spaces that support a non-constant
solution to two general classes of equations
involving the Hessian of a function and an invariant $2$-tensor.
We also consider trace-free versions of these systems. Our results
generalize earlier rigidity results for gradient Ricci solitons
and warped product Einstein metrics. In particular, our results apply to
homogeneous gradient solitons of any invariant curvature flow and
give a new structure result for homogeneous conformally Einstein
metrics. 
\end{abstract}

\maketitle

\section{Introduction}
Let $G$ be a group acting by isometries on a Riemannian manifold
$(M,g)$ and $f$ a real valued function on $M$. If $f$ is invariant
under $G$, then the Hessian of $f$ is also invariant. In this paper
we are interested in rigidity phenomena that occur when we conversely
assume that the Hessian is invariant but the function is not. We focus
on the case where $G$ acts transitively so that any invariant function
is constant. The prototypical example of a function which has invariant
Hessian but is not invariant is a linear function on $\bbR^{n}$ whose
Hessian, being zero, is invariant under the full isometry group. Another
prominent example is the restriction of coordinate functions $x^{i}$
in $\mathbb{R}^{n+1}$ to the sphere $S^{n}$, whose Hessian on the
sphere satisfies $\Hess x^{i}=-x^{i}g$. In this case, while $\Hess x^{i}$
is not invariant under the full isometry group, its trace free part
is and it also satisfies an equation of the form $\Hess x^{i}=x^{i}q$
where $q=-g$ is invariant under the isometry group. Note that the coordinate
functions in $\mathbb{R}^{n,1}$ restricted to hyperbolic space satisfy
a similar equation $\Hess x^{i}=x^{i}g$.

More complicated examples come from gradient solitons to curvature
flows. These satisfy $\Hess f=\lambda g-q$, where $q$ is an expression
involving the curvatures of the metric. Equations involving the Hessian
of a function and the curvature also come up naturally in the study
of warped products and conformal changes of metrics.

Motivated by these examples we consider the following general classes
of equations involving $q$, a symmetric two tensor on a Riemannian
manifold, 
\begin{align}
\mathrm{Hess}f & =q,\label{eqn 1.1}\\
\mathrm{Hess}w & =wq,\label{eqn 1.2}
\end{align}
where $f,w$ are smooth functions. Given a Riemannian manifold $(M,g)$
and a fixed tensor $q$ we denote by $F(M,g,q)$ and $W(M,g,q)$ the
space of all solutions to equation (\ref{eqn 1.1}) and (\ref{eqn 1.2})
respectively. We will often simply write $F(q)$ and $W(q)$.

When $q$ is fixed, equations (\ref{eqn 1.1}) and (\ref{eqn 1.2})
are overdetermined in $f$ or $w$ respectively, as there is only
one unknown function but $\frac{n(n+1)}{2}$ equations. Thus the solution
spaces $F$ and $W$ are small except in exceptional circumstances.
On the other hand, if $q$ is invariant under $G$, a group of isometries,
then $G$ acts on $F$ and $W$. Thus if $G$ is a large group we
have a large group acting on a small space and this also leads to
rigidity. Roughly speaking, this is the  approach  we use to prove general structure theorems
for any $G$-homogeneous Riemannian metric that supports non-constant
solutions to (\ref{eqn 1.1}) and (\ref{eqn 1.2}) for a $G$-invariant
$q$.

Our results build on previous work of the authors in two cases involving
the Ricci curvature. Namely, functions in $F(\lambda g-\Ric)$ corresponding
to gradient Ricci solitons and functions in $W(\frac{1}{m}(\Ric-\lambda g))$, $m\in \mathbb N$
corresponding to warped product Einstein metrics \cite{Kim-Kim, Case-Shu-Wei}. These equations
on homogeneous manifolds were studied by the authors in \cite{Petersen-Wylie-Symm}
and by the authors along with He in \cite{He-Petersen-Wylie-Homogeneous} respectively. The main idea
of this paper is that a general structure extends to the more general equations, with some important variation.

In \cite{Petersen-Wylie-Symm} the authors showed that a homogeneous
gradient Ricci soliton is the product of an Einstein metric and a
Euclidean space. We prove the following generalization of this result.

\begin{theorem} \label{Thm:Soliton} Let $(M,g)$ be a $G$-homogeneous
manifold and  $q$ a $G$-invariant symmetric two-tensor which is divergence free.
If there is a non-constant function in $F(q)$, then $(M,g)$ is a
product metric $N\times\mathbb{R}^{k}$ and $f$ is a function on
the Euclidean factor. \end{theorem}

\begin{remark} Note that $2\mathrm{div}\Ric=d\scal$, so on a homogeneous space the
Ricci tensor is divergence free. By Proposition \ref{Prop:ssplit-properties}  the divergence free assumption on $q$  can also be replaced with the assumption that $\mathrm{Ric}(\nabla f, \nabla f) \geq 0$ for $f \in F(q)$, which is also satisfied for homogeneous gradient Ricci solitons as $\Ric(\nabla f) = 0$.   \end{remark}

\begin{remark}   Griffin applies Theorem \ref{Thm:Soliton} to  study homogeneous gradient solitons for the four-dimensional Bach flow in \cite{Griffin}. 
\end{remark}

On the other hand, Theorem \ref{Thm:Soliton}
is not true if we do not assume $q$ is divergence free, see Example
\ref{example:Hyperbolic}. 
 We prove a  general structure theorem for $F(q)$
without the divergence free assumption (Theorem \ref{Thm:F}), whose precise statement we delay until section 3. The general rigidity we obtain involves spaces we call \emph{one-dimensional
extensions}.

\begin{definition} A $G$-homogeneous space $(M=G/G_{x},g)$ is called
a one-dimensional extension if there is a closed subgroup, $H \subset G$ that contains $G_x$ such that  there is a surjective Lie group homomorphism from $G$ to the additive
real numbers whose kernel is $H$. \end{definition}

The algebraic condition of being a one-dimensional extension implies
a geometric/topological product structure such that $M$ is diffeomorphic
to $\bbR\times(H/G_{x})$ and $g=dr^{2}+g_{r}$ where $g_{r}$ is
a one-parameter family of homogeneous metric on $H/G_{x}$. Moreover,
$G$ acts as a semi-direct product $G=H\rtimes\bbR$ on $g$. Theorem
\ref{Thm:F} roughly says that if $F(q)$ contains a non-constant
function then $M$ is either a one-dimensional extension, a product
of a one-dimensional extension with Euclidean space, or a space as
in Theorem \ref{Thm:Soliton}. In particular, Theorem \ref{Thm:F}
applies to any homogeneous gradient soliton for an invariant curvature
flow. We are not aware of any examples of flows where examples of gradient solitons
on one-dimensional extensions have arisen.

One-dimensional extensions play a larger role in the study of $W(q)$
as they occur even in the warped product Einstein case. In fact, in
\cite{Lafuente} Lafuente showed that a homogeneous space admits a
one-dimensional extension which is the base of a warped product Einstein manifold if and only if it is an algebraic Ricci soliton. For general $q$, we obtain
the following structure result.

\begin{theorem} \label{Thm:W intro} Let $(M,g)$ be a $G$-homogeneous
manifold and  $q$ a $G$-invariant symmetric two-tensor. If $W(q)$ is nontrivial,
then $(M^{n},g)$ is isometric to one of the following 
\begin{enumerate}
\item a space of constant curvature and $\dim W=n+1$, 
\item the product of a homogeneous space and a space of constant curvature
with $W$ consisting of functions on the constant curvature factor
and $2\leq\dim W\leq n$,
\item the quotient of the product of a homogeneous space and $\mathbb{R}$,  $(H\times \mathbb{R})/\pi_1(M)$, 
with $W=\{w:\mathbb{R}\rightarrow \mathbb{R} \mid w''=\tau w\}$ where $\tau<0$ is constant, or  
\item a one-dimensional extension and $\dim W=1$. 
\end{enumerate}
If, in addition, $q$ is Codazzi, then $(M,g)$ is isometric to one of the cases (1)-(3). 
\end{theorem}

\begin{remark} A symmetric $2$-tensor is Codazzi if its covariant derivative is symmetric, i.e. $(\nabla_X q)(Y,Z)  = (\nabla_Y q)(X,Z)$, for all vectors $X,Y,Z$.  In general, divergence free and Codazzi are different conditions.  However, a Codazzi tensor is divergence free if and only if it has constant trace.  Thus  a Codazzi tensor that is  invariant under a transitive group of isometries is  divergence free. See section 6 for further discussion of examples in case (4) where  $q$ is divergence free.
  \end{remark}

We also consider the  trace-free versions of these equations,
\begin{align}
\tag{1.1a}\mathring{\mathrm{Hess}f} & =\mathring{q},\label{eqn 1.1a}\\
\tag{1.2a}\mathring{\mathrm{Hess}w} & =w\mathring{q},\label{eqn 1.2a}
\end{align}
where $\mathring{q}$ is the trace-free part of $q$, $ \mathring{q} = q - \frac{\tr q}{\dim(M)} g$.  We write $\mathring{F}(q)$ and $\mathring{W}(q)$ for the solution
spaces to (\ref{eqn 1.1a}) and (\ref{eqn 1.2a}) respectively. Non-trivial functions in $\mathring{F}(-\Ric)$
are called Ricci almost solitons in the literature, see for example \cite{Homogen-Almost-Ricci}.
Non-trivial functions in $\mathring{W}(\frac{1}{2-n}\Ric)$ are called almost Einstein metrics in the literature.  If this case, if the function is positive then the metric is  conformal to an Einstein metric.  See, for example, \cite{DM1, DM2, Gover, Leitner, KR-ConformalEinstein} and the reference there-in.

The study of the solution spaces $\mathring{F}$
and $\mathring{W}$ can in the homogeneous case be reduced to the
study of a corresponding $F$ or $W$ space. A space of functions
$\mathring{F}$ (or $\mathring{W}$) is called \emph{essential} if
$\mathring{F}(q)\neq F(q')$ for all $q'$ (or $\mathring{W}(q)\neq W(q')$
for all $q'$).  We have the following rigidity result for essential
spaces of solutions.

\begin{theorem} \label{Thm:Essential} Let $(M,g)$ be a $G$-homogeneous
manifold and  $q$ a $G$-invariant symmetric two-tensor. If $\mathring{F}(q)$
is essential then $(M,g)$ is a space of constant curvature. If $\mathring{W}(q)$
is essential, then $(M,g)$ is locally conformally flat. \end{theorem}

Note that homogeneous locally conformally flat metrics are classified
by Takagi in \cite{Takagi} (see also Theorem \ref{thm:Takagi}).
Theorem \ref{Thm:Essential} combined with structure results for $F$
and $W$ as well as Takagi's classification yield the following corollaries.

\begin{corollary} \label{cor:oF-structure} Let $(M,g)$ be a $G$-homogeneous
manifold and  $q$ a $G$-invariant symmetric two-tensor which is divergence free.
If there is a non-constant function in $\mathring{F}(q)$, then $(M,g)$
is either a space of constant curvature or is a product metric $N\times\mathbb{R}^{k}$
with $f$ being a function on the Euclidean factor. \end{corollary}

\begin{corollary} \label{Corollary:oW-structure} If $(M,g)$ is a $G$-homogeneous
manifold and  $q$ is a $G$-invariant symmetric two-tensor such
that $\mathring{W}$ is non-trivial, then $(M,g)$ is isometric to
either 

\begin{enumerate}
\item $S^{n}(\kappa)/\Gamma$, $\bbR^{n}/\Gamma$, $H^{n}(-\kappa)$, $(S^{k}(\kappa)/\Gamma)\times H^{n-k}(-\kappa)$,
$(\bbR^{1}/\Gamma)\times H^{n-1}(-\kappa)$, or $(S^{n-1}(\kappa)\times\bbR^{1})/\Gamma$, 
\item a direct product of a homogeneous space and a space of constant curvature
with $\mathring{W}$ consisting of functions on the constant curvature
factor,
\item the quotient of the product of a homogeneous space and $\mathbb{R}$,  $(H\times \mathbb{R})/\pi_1(M)$, 
with $W=\{w:\mathbb{R}\rightarrow \mathbb{R} \mid w''=\tau w\}$ where $\tau<0$ is constant, or  
\item a one-dimensional extension of a homogeneous space. 
\end{enumerate}
Moreover, when $(M,g)$ is not in case (1), $\mathring{W}(q)=W(q')$,
where $q'$ is a $G$-invariant tensor of the form $q'=q-\lambda g$
for some $\lambda\in\mathbb{R}$. If, in addition, $q$ is Codazzi, then $(M,g)$ is isometric to one of the cases (1)-(3). \end{corollary}

In the case of  Ricci almost  solitons, Corollary \ref{cor:oF-structure} 
already follows from \cite[Theorem 1.1]{Homogen-Almost-Ricci}. For almost Einstein metrics, Corollary \ref{Corollary:oW-structure} generalizes Theorem 5.2 in \cite{Leitner} to the non-compact case. In dimension $4$, homogeneous
conformally Einstein spaces were classified in \cite{4d-Conformal} where is it
shown that if a space is not a symmetric space, then it is one of
three families of one-dimensional extensions. In higher dimensions, Corollary \ref{Corollary:oW-structure} reduces the problem of classifying homogeneous almost Einstein spaces and thus conformally Einstein spaces to studying one-dimensional extensions. We discuss this case further in  section \ref{Section-QE}, where we also
discuss the application of Corollary \ref{Corollary:oW-structure}
to more general ``generalized $m$-quasi-Einstein metrics."

As a final application of the theorems above, we consider the case
of a compact locally homogeneous manifold admitting non-trivial
functions in $F$, $\mathring{F}$, $W$, or $\mathring{W}$ for a
local isometry invariant $q$. First note that $F(q)$ can never be
non-trivial because if $f \in F(q)$  then $\Delta f = \tr q$ and $\tr q$ is constant as $q$ is a local isometry invariant tensor.  A function on a compact manifold with constant Laplacian is constant, so $f$ is constant. On the other hand,
the sphere supports invariant tensors $q$ such that  $\mathring{F}$, $W$ and $\mathring{W}$ all non-trivial.
In this case we get the following rigidity result. The proof follows
from inspecting the possibilities for simply connected examples in
Corollaries \ref{cor:oF-structure} and \ref{Corollary:oW-structure}
to admit nontrivial $\mathring{F}$, $W$ and $\mathring{W}$ that
are invariant under co-compact actions of deck transformations.

\begin{theorem} \label{Thm:CompactIntro} Suppose that $(M,g)$ is a compact locally homogeneous
manifold and $q$ a local isometry invariant symmetric two tensor. 
\begin{enumerate}
\item If $\mathring{F}(q)$ contains a non-constant function, then $(M,g)$
is a spherical space form. 
\item If $\mathring{W}(q)$ is non-trivial, then either $(M,g)$ is isometric to a direct product
of a homogeneous space $N$ and a spherical space form, isometric to $(N\times \mathbb{R})/\pi_1(M)$, or isometric
to $(S^{n-1}(\kappa)\times\bbR^{1})/\Gamma$. 
\end{enumerate}
In particular any positive function in  $\mathring{F}(q)$ or $\mathring{W}(q)$  must be constant. 
\end{theorem} Note that in the statement of part (2) we allow $N$
to be a point, so that the space could be isometric to a spherical
space form.

The paper is organized as follows.  In the next section we discuss preliminaries including the basic algebraic structure of the spaces $F$ and $W$ and the rigidity theorems for homogeneous spaces which we use to prove the structure theorems.  In the next four sections we prove the results for $F$, $\mathring{F}$, $W$, and $\mathring W$.  In the final section we discussion the application of the results to conformally Einstein and generalized $m$-Quasi Einstein metrics.  We also include an appendix with a discussion of of these spaces of functions on K\"ahler manifolds. 

\section{Preliminaries}

In this section we discuss some basic properties about the spaces
of functions $F(M,g,q)$, $\mathring{F}(M,g,q)$, $W(M,g,q)$, and
$\mathring{W}(M,g,q)$ as well as some rigidity results for homogeneous
spaces that will be the main tools in the proofs of our structure
theorems.

\subsection{Basic Structure}

First note that the spaces of functions  $F$ and $\mathring{F}$ are affine as $f_{1},f_{2}\in F$
(resp, $\mathring{F}$) implies $f_{1}-f_{2}\in V$ (resp, $\mathring{V}$
), where 
\begin{align*}
V & =\{v\mid\mathrm{Hess}v=0\}\\
\mathring{V} & =\{v\mid\mathring{\mathrm{Hess}v}=0\}.
\end{align*}
Both $V$ and $\mathring{V}$ are vector spaces of functions that
contain the constant functions. Moreover, it is well known that if
$V$ or $\mathring{V}$ contain a non-constant function, then the
metric must be special. If there is a non-constant function $v\in V$,
then $(M,g)$ must split as a product with a Euclidean factor, and
$v$ is a coordinate function in the Euclidean direction (See Proposition \ref{Prop:Vsplitting}). If there
is a non-constant function $v\in\mathring{V}$, then $(M,g)$ must
split as a warped product over a 1-dimensional base.  This was first
proven locally by Brinkmann \cite{Brinkmann} and later globally by
Tashiro \cite{Tashiro}. The complete study of the full space $\mathring{V}$
is due to Osgood-Stowe \cite{OS}.

The spaces $W$ and $\mathring{W}$ are vector spaces of functions.
In fact, note that $V$ and $\mathring{V}$ are special cases of $W$
and $\mathring{W}$ where $q=0$. Rigidity for metrics which admit
linearly independent solutions in $W$ was studied in \cite{He-Petersen-Wylie} (Also see Theorem \ref{Thm:WPR} below).
It gives a weaker warped product splitting than for $V$ or $\mathring{V}$.

A tensor $q$ is invariant under a subgroup, $G$, of isometries
of $(M,g)$, if  $\gamma^{*}q=q$ for all $\gamma \in G$.  If $q$ is invariant under $G$, then $G$ acts on the spaces $F$, $\mathring{F}$,
$W$, and $\mathring{W}$ via $f\mapsto\gamma^{*}f$, $\gamma\in G$.
Conversely, we also have that if $F$ or $W$ is invariant under the action of $G$ then so is $q$. 

\begin{proposition} \label{Prop:q-vs-W-inv} If $F(M,g,q)$ or $W(M,g,q)$
are nontrivial and invariant under the action of $G\subset\mathrm{Isom}(M,g)$,
then $q$ is also invariant under G. \end{proposition}

\begin{proof} We consider the case where $W$ is invariant. The case
for $F$ is similar. Fix a nontrivial $w\in W$ and $\gamma\in G$.
We have:

\begin{align*}
(w\circ\gamma)q & =\mathrm{Hess}(w\circ\gamma)=\gamma^{*}\mathrm{Hess}w=\gamma^{*}(wq)=(w\circ\gamma)(\gamma^{*}q).
\end{align*}
This shows that $\gamma^{*}q=q$ wherever $w\circ\gamma\neq0$. Since
this is a set of full measure unless $w\equiv0$ (see \cite[Proposition 1.1]{He-Petersen-Wylie}) we conclude that
$q$ is $\gamma$ invariant. \end{proof}

\subsection{Some rigidity results on homogeneous spaces}

In this section we discuss some rigidity results for certain functions
and vector fields on homogeneous spaces.  We first recall the algebraic formulation
of the rigidity we require  from the introduction. 

\begin{definition} A $G$-homogeneous space $(M=G/G_{x},g)$ is called
a one-dimensional extension if there is a closed subgroup, $H \subset G$ that contains $G_x$ such that  there is a surjective Lie group homomorphism from $G$ to the additive
real numbers whose kernel is $H$. \end{definition}

This algebraic property has the following geometric consequences.

\begin{proposition} \label{Def:Semi-Split} If a $G$-homogeneous
space $(M=G/G_{x},g)$ is a one-dimensional extension of $H$, then 
\begin{enumerate}
\item $G$ acts on $M$ as a semi-direct product group $G=H\rtimes\mathbb{R}$. 
\item $M$ is diffeomorphic to $\left(H/G_{x}\right)\times\mathbb{R}$, 
\item $g=g_{r}+dr^{2}$ where $g_{r}$ is a one-parameter family of homogeneous
metrics on $H/G_{x}$, 
\end{enumerate}
\end{proposition}

\begin{proof} Let $\phi:G\to\mathbb{R}$ be a surjective Lie group
homomorphism with kernel $H$. Since $G_{x}\subset H$ it follows
that $M/H=(G/G_{x})/H=G/H=\mathbb{R}$. Therefore, the action of $H$
on $M$ has cohomogeneity one. Let $r:M\to M/H$. By re-parametrizing
the range, $M/H$, we can assume that $r$ is a distance function.
$H$ acts transitively on the level sets of $r$, which gives the
diffeomorphic splitting (2) as well as the metric of the form (3).

To see (1), let $\gamma_{t}$ be a one-parameter family of isometries
in $G$. It follows that $t\mapsto\phi(\gamma_{t})$ is an additive
group homomorphism from $\mathbb{R}$ to $\mathbb{R}$ and thus either
trivial or an isomorphism. Since $\phi$ is assumed to be surjective,
we can find a $\gamma_{t}$ such that this map is an isomorphism.
Let $\gamma\in G$. There is $t$ such that $\phi(\gamma_{-t})=\phi(\gamma)$,
which implies that $\gamma_{t}\circ\gamma\in H$. This shows that
$G$ is a semi-direct product group $G=H\rtimes\mathbb{R}$.

\end{proof}
Now we are ready to prove the main Lemma which we use to show that spaces are one-dimensional extensions.  It roughly says that when there is function which is ``almost" invariant by a transitive
group in the sense that it changes only by an additive or multiplicative
constant, then we obtain a one-dimensional extension.

\begin{lemma} \label{Lem:ssplit-rigid} Let $M$ be a $G$-homogeneous
space, assume that either 
\begin{enumerate}
\item there is a non-constant function $f$ such that for all $\gamma\in G$
there is $C_{\gamma}\in\mathbb{R}$ so that 
\[
\gamma^{*}f=f+C_{\gamma},\textrm{or}
\]
\item there is a non-constant function $w$ such that for all $\gamma\in G$
there is $C_{\gamma}\in\mathbb{R}$ so that 
\[
\gamma^{*}w=C_{\gamma}w.
\]
\end{enumerate}
In either case $(M,g)$ becomes a one-dimensional extension of $H$,
the subgroup of $G$ that fixes the function $f$ or $w$. Moreover,
in case (1) $f=ar+b$ and in case (2) $w=be^{ar}$ for some $a,b\in\mathbb{R}$.
\end{lemma}

\begin{proof} First consider case (1). The assumption $\gamma^{*}f=f+C_{\gamma}$,
gives a homomorphism $\gamma\mapsto C_{\gamma}$ into the additive
real numbers with kernel $H=\{\gamma\in G\mid\gamma^{*}f=f\}$. To
see that $G_{x}\subset H$ note that if $\gamma(x)=x$, then $\gamma^{*}f(x)=f(x)$
implying that $C_{\gamma}=0$. Observe that the image of $\gamma\mapsto C_{\gamma}$
is either trivial or $\bbR$ and in case it is trivial $f$ is forced
to be constant. Therefore, we have a one-dimensional extension of
$H$ and the diffeomorphic splitting $M=H/G_{x}\times\mathbb{R}$
with metric $g=g_{r}+dr^{2}$. As $f$ is invariant under $H$ we
must have $f=f(r)$, $\nabla f=f'(r)\nabla r$. Since the group $G$
preserves $\nabla f$ this implies that $f'(r)$ is constant, so $f=ar+b$
for a constants $a,b\in\mathbb{R}$. This completes case (1).

Case (2) is similar. Since $\gamma^{*}w=C_{\gamma}w$, the action
of $G$ preserves both the zeros and the critical points of $w$.
Since $G$ is transitive and $w$ is non-constant we must have that
$w$ has no zeros nor critical points so, by possibly switching to $-w$, we can assume
that $w$ is positive. The map $\gamma\mapsto C_{\gamma}$ is a group
homomorphism into the multiplicative group of positive real numbers.
But then $\mathrm{ln}(C_{\gamma})$ gives a homomorphism into the
additive reals whose kernel consists of the isometries that preserve
$w$. We then obtain $M=H/G_{x}\times\mathbb{R}$ with metric $g=g_{r}+dr^{2}$
and $w=w(r)$.

To see that $w=be^{ar}$ consider that any isometry $\gamma$ preserves
the vector field $\frac{\nabla w}{w}$ as 
\[
d\gamma\left(\frac{\nabla w}{w}(\gamma^{-1}x)\right)=\frac{d\gamma(\nabla w(\gamma^{-1}x))}{w(\gamma^{-1}x)}=\frac{C_{\gamma}\nabla w(x)}{C_{\gamma}w(x)}=\frac{\nabla w}{w}(x).
\]
So $|\nabla w|/w=|w'(r)|/w(r)$ is constant and so $w=be^{ar}$ for
some $a,b\in\mathbb{R}$. \end{proof}


Finally in this section we prove a fact about conformal fields on
homogeneous spaces. Recall that a vector field $X$ is a conformal
field if $\mathring{L_{X}g}=0$ which is equivalent to the 1-parameter
family of (local) diffeomorphisms generated by $X$ being conformal diffeomorphisms
of $g$. We have the following rigidity for conformal fields on homogeneous
spaces. This result was established and used in \cite[Proof of Theorem 1.1]{Homogen-Almost-Ricci}, 
but the resulting formula there does not appear to be entirely correct.

\begin{proposition} \label{Prop:CField} Let $(M,g)$ be a homogeneous
space and $X$ a conformal field, then either $(M,g)$ is locally
conformally flat, or $X$ is a Killing field. \end{proposition}

\begin{proof} All two-dimensional spaces are locally conformally
flat, so there is nothing to prove in this case. In dimensions larger
than $2$ there is always a conformally invariant $(1,3)$ tensor, $C$,
on $(M,g)$ such that $C=0$ if and only if $(M,g)$ is locally conformally flat. In dimension $3$ it is the Cotton tensor, in higher
dimensions the Weyl tensor.  

The conformal invariance of $C$ implies that $L_{X}C=0$ as $X$ is
a conformal field. We claim that $D_{X}|C|^{2}=-2\mathrm{tr}(L_{X}g)|C|^{2}$.
To see this consider a point $p\in M$ where $V(p)\neq0$ and select
coordinates $x^{1},\dots,x^{n}$ such that $V=\partial_{1}$. The
Lie derivative of any tensor can now be calculated by computing the
directional derivatives of the components of the tensor in these coordinates.
With this in mind it follows that the components of the metric tensor
satisfy: $D_{X}g_{ij}=\mathrm{tr}(L_{X}g)g_{ij}$ and its inverse:
$D_{X}g^{ij}=-\mathrm{tr}(L_{X}g)g^{ij}$, while $D_{X}C_{ijk}^{l}=0$.
We can now calculate

\begin{align*}
D_{X}|C|^{2} & =D_{X}(g^{is}g^{jt}g^{ku}g_{lv}C_{ijk}^{l}C_{stu}^{v})\\
 & =(-3\mathrm{tr}(L_{X}g)+\mathrm{tr}(L_{X}g))(g^{is}g^{jt}g^{ku}g_{lv}C_{ijk}^{l}C_{stu}^{v})\\
 & =-2\mathrm{tr}(L_{X}g))(g^{is}g^{jt}g^{ku}g_{lv}C_{ijk}^{l}C_{stu}^{v}).
\end{align*}

Finally, the formula trivially holds on any open set where $X$ vanishes.
(In fact, a non-trivial conformal field cannot vanish on an open set as its zero
set has components that are either points or totally umbilic hypersurfaces.)
So the formula $D_{X}|C|^{2}=-2\mathrm{tr}(L_{X}g)|C|^{2}$ must hold
globally.

Since the space is homogeneous, $|C|^{2}$ is constant, so either
$\mathrm{tr}(L_{X}g)=0$ everywhere, and the field is Killing, or
there is a point where $|C|^{2}=0$. However, again by homogeneity,
if $C=0$ at a point then $C=0$ everywhere and then the space is locally
conformally flat. \end{proof}

Finally in this section we point out that locally conformally flat
homogeneous spaces have a rigid classification due to Takagi.

\begin{theorem}\cite[Theorem B]{Takagi} \label{thm:Takagi}Let $(M^{n},g)$ be a
homogeneous space which is locally conformally flat, then $(M,g)$
is isometric to either $S^{n}(\kappa)/\Gamma$, $\bbR^{n}/\Gamma$,
$H^{n}(-\kappa)$, $(S^{k}(\kappa)/\Gamma)\times H^{n-k}(-\kappa)$,
$(\bbR^{1}/\Gamma)\times H^{n-1}(-\kappa)$, or $(S^{n-1}(\kappa)\times\bbR^{1})/\Gamma$.
\end{theorem}

\section{$F$} \label{Section:F}

Now we begin the study of the space of solutions to (\ref{eqn 1.1}),  $F(M,g,q)$.  We start by offering two examples of spaces that typify situations
where $q$ is invariant under a group of isometries but not all the
functions in $F\left(q\right)$ are.

\begin{example} \label{example:Euclidean} Let $f:\mathbb{R}^{n}\to\mathbb{R}$
such that $f(x)=\frac{A}{2}|x|^{2}+L(x)+c$ where $A,C$ are constants
and $L:\mathbb{R}^{n}\to \mathbb{R}$ is a linear function. Then $\mathrm{Hess}f=Ag_{0}$
where $g_{0}$ denotes the Euclidean dot product. Clearly $\mathrm{Hess}f$
is invariant under the full isometry group, but $f$ is not. \end{example}

\begin{example} \label{example:Hyperbolic} Let $g=dr^{2}+e^{2kr}g_{0}$,
where $g_{0}$ is the Euclidean metric on $\mathbb{R}^{n-1}$. Then
$g$ is the Euclidean metric if $k=0$ and is Hyperbolic space if
$k\neq0$. Consider $f=cr$ and $G=\{\phi\mid\phi(r,x)=(r+a,e^{-ka}\tau(x)),$
where $a\in\mathbb{R}$ and $\tau\in\mathrm{Isom}(\mathbb{R}^{n-1})\}$.
In this case $G$ is a group of isometries of $g$ that acts transitively
and $\mathrm{Hess}f=cke^{2kr}g_{0}$ which is invariant under the
group $G$. \end{example}

Our results come from considering the cases when the dimension of
$V$ is one and  larger than one separately. When the dimension is one we have an almost trivial
action of a transitive group of isometries while, when the dimension is larger
than one, we have a rigidity result for the metric. Example \ref{example:Hyperbolic}
is in the case where $V$ is one dimensional and Example \ref{example:Euclidean}
is in the case where $V$ is higher dimensional.

Let us now be more precise. First in the case where $\mathrm{dim}(V)=1$,
we can apply Lemma \ref{Lem:ssplit-rigid}.

\begin{proposition} \label{Prop:C} Let $(M,g)$ be a $G$-homogeneous
manifold and let $q$ be a $G$-invariant symmetric two tensor.  If  $\mathrm{dim}(V)=1$ and $f \in F(q)$ is non-constant,
then $(M,g)$ is a one-dimensional extension and $f=kr$. \end{proposition}

\begin{proof} Recall that $\gamma^{*}f=f\circ\gamma^{-1}$. Since
$q$ is invariant under $\gamma$ we have $\gamma^{*}f\in F$. Therefore,
$\gamma^{*}f-f\in V$ and this is a real number since $V$ consists
only of constants. This shows that $\gamma^{*}f=f+C_{\gamma}$ for
a constant $C$, so we can apply Lemma \ref{Lem:ssplit-rigid}. \end{proof}

The rigidity statement for complete spaces which have non-constant
functions in $V$ is the following.

\begin{proposition} \label{Prop:Vsplitting}Suppose $(M,g)$ is a complete Riemannian manifold
and suppose that $\mathrm{dim}(V)=k+1$ for some $k\geq1$, then $M$
splits isometrically as $\mathbb{R}^{k}\times N$ for some space $N$
and $\mathrm{Isom}(M)=\mathrm{Isom}(\mathbb{R}^{k})\times\mathrm{Isom}(N)$.
Moreover, $\dim(V(N))=1$ and $V(M)$ consists of the space of affine
functions $\mathbb{R}^{k}\to\mathbb{R}$. \end{proposition}

\begin{proof} The metric splitting follows from the fact that all
elements in $V$ have parallel gradient. Moreover, $\bbR^{k}$ must
be the Euclidean de Rham factor as otherwise $\dim V>k+1$. This shows
that the isometry group splits. Finally if $\dim (V\left(N\right))>1$,
then also $\dim V>k+1$. \end{proof}

The previous two propositions show that if $f\in F(q)$ is a non-constant
function and $q$ is invariant under a transitive group of isometries, then
the metric is either a one-dimensional extension or splits as a product.
In the case of a product splitting, we do not assume that the tensor
$q$ necessarily splits, however a further application of Lemma \ref{Lem:ssplit-rigid}
allows us to determine the function $f$ when the metric splits.

\begin{proposition} \label{Prop:fsplit} Let $M=B\times F$ be a
direct product and let $G=G_{1}\times G_{2}$ where $G_{1}$, $G_{2}$
are transitive groups of isometries on $B$ and $F$ respectively.
Suppose that there is a function $f$ on $B\times F$ such that 
\[
(\gamma^{*}f)(x,y)-f(x,y)=\phi_{\gamma}(y)
\]
for all $\gamma\in G$, where $\phi$ is a function of $F$ that depends
on $\gamma$. Either 
\begin{enumerate}
\item $f=\psi(y)$ , or 
\item $B$ is a one-dimensional extension, $g_{B}=dr^{2}+g_{r}$, and $f=ar+\psi(y)$ 
\end{enumerate}
where $\psi$ is a function of $F$ . \end{proposition}

\begin{proof} Fix a point $y_{0}\in F$, and let $f_{0}:B\times\{y_{0}\}\to\mathbb{R}$
be defined as $f_{0}(x)=f(x,y_{0})$. Let $\gamma_{1}\in G_{1}$,
by assumption we have 
\begin{align*}
((\gamma_{1}\times\mathrm{id})^{*}f)(x,y_{0})-f(x,y_{0}) & =\phi_{1}(y_{0}),\\
((\gamma_{1})^{*}f_{0})(x)-f_{0}(x) & =\phi_{1}(y_{0}).
\end{align*}
So, applying Lemma \ref{Lem:ssplit-rigid} we get that either $f_{0}$
is constant in $x$ or $B\times\{y_{0}\}$ is a one-dimensional extension
and $f_{0}=a(0)r+b(0)$.

If $f_{0}(x)=d$ for a constant $d$, then let $\gamma_{2}\in G_{2}$
and consider 
\begin{align*}
((\mathrm{id}\times\gamma_{2})^{*}f)(x,y_{0})-f(x,y_{0}) & =\phi_{2}(y_{0}),\\
f(x,\gamma_{2}(y_{0}))-d & =\phi_{2}(y_{0}).
\end{align*}
Since $G_{2}$ acts transitively, this implies that $f$ is constant
in the $x$ direction everywhere.

On the other hand, if $f_{0}$ is non-constant and $B\times\{y_{0}\}$
is a one-dimensional extension, then $B\times\{y\}$ is a one-dimensional
extension for all $y$ since $M$ is assumed to be a product metric.
Applying Lemma \ref{Lem:ssplit-rigid} to each $f_{y}(x)=f(x,y)$
we obtain that $f(x,y)=a(y)r+b(y)$ where $a,b$ could a priori be
functions of $y$. But then $a$ must be constant as 
\begin{align*}
(\mathrm{id}\times\gamma_{2})^{*}(f)(x,y_{0})-f(x,y_{0}) & =((\gamma_{2}^{*}a)(y_{0})-a(y_{0}))r+(\gamma_{2}^{*}b)(y_{0})-b(y_{0}).
\end{align*}
Since the right hand side is assumed to only be a function of $y$
it follows that $(\gamma_{2}^{*}a)(y_{0})=a(y_{0})$ for all $\gamma_{2}\in G_{2}$
and $a$ is constant.

\end{proof}

This gives us the following theorem.

\begin{theorem} \label{Thm:F} Let $(M,g)$ be a $G$-homogeneous
manifold and let $q$ be a $G$-invariant symmetric two tensor. Suppose that  $f\in F(q)$ is a non-constant function
 then either 
\begin{enumerate}
\item $(M,g)$ is isometric to a product, $N\times\mathbb{R}^{k}$ where
$f$ is constant on $N$, 
\item $(M,g)$ is a one-dimensional extension, $g=dr^{2}+g_{r}$, and $f(x,y)=ar+b$
, or 
\item $(M,g)$ is isometric to a product, $N\times\bbR^{k}$ where $N$
is a one-dimensional extension and $f(x,y)=ar(x)+v(y)$ where $v$
is a function on $\mathbb{R}^{n}$ and $r$ is a distance
function on $N$. 
\end{enumerate}
\end{theorem}

\begin{proof} We have already seen that the theorem is true when
$\mathrm{dim}(V)=1$. So suppose $\mathrm{dim}(V)>1$ and note that
the metric splits as a direct product, $N\times\mathbb{R}^{k}$. Moreover,
$G=G_{1}\times G_{2}$ because unit tangent vectors
to the $\bbR^{k}$ factor are characterized as gradients to functions
in $V$. We also have that $\gamma^{*}f-f$ is a function of the $\mathbb{R}^{k}$
factor for any $\gamma$. So we may apply Proposition \ref{Prop:fsplit}
to obtain the result. \end{proof}

%


The natural question coming from Theorem \ref{Thm:F} is what conditions imply that a one-dimensional extension is a product, the next proposition gives two such conditions. 

\begin{proposition} \label{Prop:ssplit-properties} Let $(M,g)$
be a one-dimensional extension. The following properties hold: 
\begin{enumerate}
\item $\Delta r$ is constant, 
\item $\mathrm{Ric}(\nabla r,\nabla r)\leq0$, 
\item If $\mathrm{Ric}(\nabla r,\nabla r)=0$, then $g=g_{0}+dr^{2}$ is
a product, 
\item If $\mathrm{div}(\nabla\nabla r)=0$, then $g=g_{0}+dr^{2}$ is a
product. 
\end{enumerate}
\end{proposition}

\begin{proof} The transitive group $G$ preserves $\nabla r$ and
$\nabla\nabla r$ is invariant by $G$ so $\Delta r=\mathrm{tr}(\nabla\nabla r)$
is constant.


To see (2) and (3) consider the Bochner formula applied to $r$: 
\[
\frac{1}{2}\Delta|\nabla r|^{2}=\mathrm{Ric}(\nabla r,\nabla r)+|\mathrm{Hess}r|^{2}+g(\nabla\Delta r,\nabla r).
\]
Since $|\nabla r|$ and $\Delta r$ are constant, we obtain 
\[
\mathrm{Ric}(\nabla r,\nabla r)=-|\mathrm{Hess}r|^{2}.
\]
So if $\mathrm{Ric}(\nabla r,\nabla r)=0$ then $|\mathrm{Hess}r|^{2}=0$,
which implies that $M$ splits isometrically as $N\times\mathbb{R}$.

Finally, for (4) note that 
\[
\mathrm{div}(\nabla\nabla r)=\nabla\Delta r+\mathrm{Ric}(\nabla r).
\]
So, as $\Delta r$ is constant, the condition $\mathrm{div}(\nabla\nabla r)=0$
implies that $\mathrm{Ric}(\nabla r)=0$ and we have a product splitting.
\end{proof}



This allows us to prove Theorem \ref{Thm:Soliton}

\begin{proof}[Proof of Theorem \ref{Thm:Soliton}] We have that
$\mathrm{Hess}f=q$ for a tensor $q$ that is invariant under a transitive
group of isometries.  Assume that $f$ is
non-constant, then since $q$ is invariant under isometries Theorem
\ref{Thm:F} implies that either $M$ is a one-dimensional extension
or $M$ splits as a product metric $M=N\times\mathbb{R}^{k}$, $g=g_{1}+g_{2}$.
Assume also that this splitting is maximal in the sense that $M$
does not split off more than $k$ Euclidean factors.

If $\mathrm{div}(q)=0$ then  $\mathrm{div}(\nabla\nabla r)=0$, so by Proposition
\ref{Prop:ssplit-properties} the one-dimensional extension in the splitting is itself a product
$\mathbb{R}\times N$, where $r$ is the coordinate in the $\mathbb{R}$
direction. But, this contradicts the maximality of the splitting. 

Therefore, we have $M=N\times\mathbb{R}^{k}$, by Theorem \ref{Thm:F}
we also have a splitting of the function $f$ of the form $f=f_{1}\times f_{2}$
where $f_{1}$ is a function on $N$ and $f_{2}$ is a function on
$\mathbb{R}^{k}$. In particular, $q=\mathrm{Hess}f=\mathrm{Hess}(f_{1})+\mathrm{Hess}(f_{2})$
so $q$ splits as $q_{1}+q_{2}$ where $q_{1}$ is a tensor on $N$
and $q_{2}$ is a tensor on $\mathbb{R}^{k}$. In particular, $\mathrm{div}q=\mathrm{div}(q_{1})+\mathrm{div}(q_{2})$,
so $\mathrm{div}(q_{1})=0$. If $f_{1}$ is non-constant then, by
Theorem \ref{Thm:F}, $(N,g_{1})$ is a one-dimensional extension
with $\mathrm{Hess}f_{1}=q_{1}$ and $\mathrm{div}q_{1}=0$. So we
also obtain that the one-dimensional extension in this case is a product,
again contradicting the maximality of the splitting. Therefore, for
the maximal splitting, we must have that $f_{1}$ is constant on the
$N$ factor. \end{proof} 

\section{Traceless $F$}

Now we consider spaces of functions $\mathring{F}(M,g,q)$ of solutions to (\ref{eqn 1.1a}).
Given our established results about the corresponding space $F(M,g,q)$,
we consider the question of when $\mathring{F}(M,g,q)\neq F(M,g,q)$.
There is a trivial way to produce such examples by adding a factor
of $g$ to $q$. Namely, if $f\in F(M,g,q-\phi g)$ for $\phi\in C^{\infty}(M)$, $\phi \neq 0$, 
then $f\notin F(M,g,q)$, but $f\in\mathring{F}(M,g,q)$. This motivates
the following definition.

\begin{definition} Let $(M,g)$ be a Riemannian manifold and $q$
a symmetric two-tensor, then $\mathring{F}(M,g,q)$ is inessential if $\mathring{F}(M,g,q)=F(M,g,q')$
for some quadratic form $q'$. $\mathring{F}(M,g,q)$ is essential
if it is not inessential. \end{definition}

The next proposition shows that essential spaces are easily characterized
in terms of the spaces $\mathring{V}$ and $V$. It also shows that
the property of $\mathring{F}$ being essential is a property of the
space $(M,g)$ but not the choice of $q$.


\begin{proposition} \label{Prop:TF-simple} Let $(M,g)$ be a Riemannian manifold and $q$
a symmetric two-tensor, then the following are equivalent: 
\begin{enumerate}
\item $\mathring{F}(M,g,q)$ is essential, 
\item $\mathring{F}(M,g,q)\neq F(M,g,q-\phi g)$ for all $\phi\in C^{\infty}(M)$, 
\item The map $\Delta:\mathring{F}(M,g,q)\to C^{\infty}(M)$ is non-constant,
and 
\item $\mathring{V}\neq V$. 
\end{enumerate}
Moreover, if $\mathring{F}(M,g,q)=F(M,g,q')$ is inessential and $q$
is invariant under $G\subset\mathrm{Isom}(M,g)$ then $q'$ is also
invariant under $G$. \end{proposition}

\begin{proof}

(1) $\Rightarrow$ (2) is obvious. To see (2) $\Rightarrow$ (1) consider that if (1) is not true
then $\mathring{F}(M,g,q)=F(M,g,q')$.  So $\Hess f=q'$ and 
\[
\mathring{q}=\mathring{\Hess f}=q'-\frac{\tr(q')}{n}g.
\]
So $q'=q+\frac{\tr(q')-tr(q)}{n}g$ which would contradict (2). 

(1) and (3) are equivalent because if two quadratic forms have the
same trace free part, then they are the same if and only if they have
the same trace.

To see that (3) and (4) are equivalent note that $w\in\mathring{V}$
is an element of $V$ if and only if $\Delta w=0$. If $f,f'\in\mathring{F}(M,g,q)$
then $f-f'\in\mathring{V}$, so $\Delta$ being non-constant on $\mathring{F}(M,g,q)$
is equivalent to there being a function in $\mathring{V}$ with non-zero
Laplacian.

The final statement follows from Proposition \ref{Prop:q-vs-W-inv}.
\end{proof}

The next example shows that for simply connected spaces of constant
curvature, $\mathring{F}$ is essential.

\begin{example} Let $(M^{n},g)$ be a simply connected space of constant
curvature. Then $\mathrm{dim}(\mathring{V})=n+2$ and $\mathring{F}(M,g,q)$ is essential. If $(M^{n},g)$ is Euclidean
space then $V$ is the $n+1$ dimensional space of affine functions and  $\mathring{V}$ is spanned by $V$ along with the function $|x|^{2}$. If $M^{n}$ is a sphere or hyperbolic space then $V$ just contains constant functions.  For the sphere $\mathring{V}$
also contains the restriction the coordinate functions in $\mathbb{R}^{n+1}$
while  for  hyperbolic
space $\mathring{V}$ contains the restriction of the coordinate
functions in $\mathbb{R}^{1,n}$.  See \cite{He-Petersen-Wylie} for more details.  \end{example}

On the other hand $\mathring{F}$ is inessential for product spaces.

\begin{proposition} \label{Prop:Products} If $(M,g)=(M_{1}^{n_{1}}\times M_{2}^{n_{2}},g_{1}+g_{2})$,
then $\mathring{V}=V$, so $\mathring{F}$ is inessential. \end{proposition}

\begin{proof}

Consider $f(x_{1},x_{2})\in\mathring{V}$. Then $\mathrm{Hess} f (X,U) = 0$ for $X \in TM_1$ and $U \in TM_2$ so by \cite[Lemma 2.1]{Petersen-Wylie-Symm}
 $f(x_{1},x_{2})=f_{1}(x_{1})+f_{2}(x_{2})$. Thus 
\[
\hess_{g}(f)=\hess_{g_{1}}f_{1}+\hess_{g_{2}}f_{2}=\frac{\Delta_{g_{1}}f_{1}+\Delta_{g_{2}}f_{2}}{n}g.
\]
If we restrict this equation to $M_{1}$ and $M_{2}$ this tells us
that $\mathring{\hess_{g_{1}}f_{1}}=0$ and $\mathring{\hess_{g_{2}}f_{2}}=0$.
Thus 
\[
\frac{\Delta_{g_{1}}f_{1}+\Delta_{g_{2}}f_{2}}{n}=\frac{\Delta_{g_{1}}f_{1}}{n_{1}}=\frac{\Delta_{g_{2}}f_{2}}{n_{2}},
\]
which shows that $\Delta_{g_{1}}f_{1}=0$ and $\Delta_{g_{2}}f_{2}=0$.
Consequently, $\mathring{V}=V$.

\end{proof}

This gives us the following characterization of essential $\mathring{F}$
in the homogeneous case.

\begin{theorem} \label{Thm:EssentialF} Suppose that $(M,g)$ is
a homogeneous Riemannian manifold. If $\mathring{F}(M,g,q)$ is essential,
then $(M,g)$ is a space of constant curvature. \end{theorem}

\begin{proof}

Suppose that $\mathring{F}$ is essential. Let $w\in\mathring{V}\neq V$,
then $\nabla w$ is a conformal field which is not Killing. By Proposition
\ref{Prop:CField}, $(M,g)$ is locally conformally flat. By Takagi,
the universal cover of $M$ is either a space of constant curvature
or a product of spaces of constant curvature. Note that if $\pi:\widetilde{M}\to M$
is the universal cover of $M$, $w\in\mathring{V}(M)$ implies $(w\circ\pi)\in\mathring{V}(\widetilde{M})$
and $v\in{V}(M)$ implies $(v\circ\pi)\in V(\widetilde{M})$. Therefore,
if $\mathring{V}(M)\neq V(M)$ then $\mathring{V}(\widetilde{M})\neq V(\widetilde{M})$,
so $M$ essential implies that $\widetilde{M}$ is. Then by Proposition
\ref{Prop:Products}, the universal cover does not split as a product
and so must be a space of constant curvature.

\end{proof}

\begin{theorem}\label{Thm:TF} Let $(M,g)$ be a $G$-homogeneous
Riemannian manifold  and  $q$ be a $G$-invariant symmetric two-tensor. If $f\in\mathring{F}$ is a non-constant
function then
either 
\begin{enumerate}
\item $(M,g)$ is a space of constant curvature, 
\item $(M,g)$ is isometric to a product, $N\times\mathbb{R}^{k}$ where
$f$ is constant on $N$, 
\item $(M,g)$ is a one-dimensional extension, $g=dr^{2}+g_{r}$, and $f(x,y)=ar+b$
, or 
\item $(M,g)$ is isometric to a product, $N\times\bbR^{k}$ where $N$
is a one-dimensional extension and $f(x,y)=ar(x)+v(y)$, where $v$
is a function on $\mathbb{R}^{n}$ and $r$ is a distance
function on $N$. 
\end{enumerate}
\end{theorem}

\begin{proof} If $\mathring{F}$ is essential, then by Theorem \ref{Thm:EssentialF}
$(M,g)$ is a space of constant curvature. If $F$ is inessential,
then $\mathring{F}(q)=F(q')$ where $q'$ is also invariant by $G$,
then Theorem \ref{Thm:F} implies the result.

\end{proof}


This allows us to prove Corollary \ref{cor:oF-structure}

\begin{proof}[Proof of Corollary \ref{cor:oF-structure}] By Theorem
\ref{Thm:EssentialF} either $(M,g)$ is constant curvature or $\mathring{F}(q) = F(q')$ and by (2) of Proposition \ref{Prop:TF-simple} $q' = q - \phi g$ for a function $\phi$.  But then since $q$ and $q$ are both invariant by the transitive group $G$ we must have $\phi$ constant. In particular, $\mathrm{div}(q') = \mathrm{div}(q)$, so $q'$ is also divergence free and the Corollary follows from applying Theorem \ref{Thm:Soliton} to $F(q')$.


\end{proof}

%

\section{$W$}

Now we consider the space $W(M,g,q)$ of solutions to equation (\ref{eqn 1.2}).
When this is a one-dimensional space we have the following
statement.

\begin{theorem} \label{Thm:Wdim1} Let $(M,g)$ be a $G$-homogeneous
manifold and let $q$ be a non-zero $G$-invariant two-tensor. If $\mathrm{dim}(W)=1$,
then $(M,g)$ is a one-dimensional extension
and $W=\{be^{ar}\mid b\in\mathbb{R}\}$. \end{theorem}

\begin{proof} Let $G$ be a transitive group of isometries and $w$
be a non-constant function in $W$. Since $W$ is one-dimensional
and $G$ acts on $W$, for $\gamma\in G$, we have  $w\circ\gamma=C_{\gamma}w$
for some constant $C_{\gamma}$. The theorem now follows from Lemma
\ref{Lem:ssplit-rigid}. \end{proof}

When $\mathrm{dim}(W)>1$ we have the following result of He-Petersen-Wylie.

\begin{theorem} \label{Thm:WPR} \cite[Theorem A and B and Proposition 6.5]{He-Petersen-Wylie}
Suppose $(M,g)$ is a complete Riemannian manifold
such that $\mathrm{dim}(W)=k+1$, $k\geq1$. If $k>1$ or $M$ is simply connected, then $M$ is isometric
to a warped product $B\times_{u}F$ where $F$ is a space of constant
curvature. Moreover, 
\[
W=\{w(x,y)=u(x)v(y)\mid v\in W(F,-\tau g_{F})\}
\]
If $k=1$, then $M$ is isometric to $(B\times_u \mathbb{R})/\pi_1(M) $, where $u>0$ and $\pi_1(M)$ acts by translations on $\mathbb{R}$.
\end{theorem}

Before applying these theorems, we need some basic results about warped
products which are homogeneous.

By a warped product, $M=B\times_{u}F$ we mean a metric of the form
$g_{M}=g_{B}+u^{2}g_{F}$ where $u:B\to\mathbb{R}$. In general, it
is possible to obtain a smooth metric $g_{M}$ even in case $u$ vanishes
on the boundary of $B$. However, in this paper we will be able to
conclude that $u>0$ and $M$ is diffeomorphic to $B\times F$. Let
$\gamma$ be a map of $B\times_{u}F$, we will say that $\gamma$
respects the warped product splitting if $\gamma=\gamma_{1}\times\gamma_{2}$
with $\gamma_{1}:B\to B$ and $\gamma_{2}:F\to F$. A group of isometries
is said to respect the splitting if all its elements do. We have the
following simple result about the isometries of a warped product that
respect the splitting.

\begin{proposition}  \cite[Lemma 5.1]{He-Petersen-Wylie-Unique}\label{Prop:IsomWP} Suppose $M=B\times_{u}F$
with $u>0$, then a map $\gamma$ which respects the splitting is
an isometry of $g_{M}$ if and only if (1) $\gamma_{1}\in\mathrm{Isom}(g_{B})$,
(2) there is a $C\in\mathbb{R}^{+}$ such that $\gamma_{1}^{*}(u)=Cu$,
and (3) $\gamma_{2}$ is a $C$-homothety of $g_{F}$. \end{proposition}

Let $\mathrm{Isom}(B)_{u}$ be the isometries of $g_{B}$ that preserve
$u$. Proposition \ref{Prop:IsomWP} implies that $\mathrm{Isom}(B)_{u}\times\mathrm{Isom}(F)$
is a group of isometries that respects the splitting. Recall also
that a complete Riemannian manifold admits a $C$-homothety with
$C\neq1$ if and only if it is a Euclidean space. Therefore, if $F$
is not a Euclidean space, then any subgroup of isometries that preserves
the splitting is a subgroup of $\mathrm{Isom}(B)_{u}\times\mathrm{Isom}(F)$.
In general, a warped product can have isometries that do not respect
the splitting, so we will have to justify this assumption when we
apply the Proposition below.


Combining Proposition \ref{Prop:IsomWP} with Lemma \ref{Lem:ssplit-rigid}
gives us the following characterization of when a warped product admits
a transitive group of isometries which preserves the splitting.

\begin{lemma} \label{Lemma:WPHomogeneous} Let $M=B\times_{u}F$
with $u>0$ be a warped product manifold which admits a transitive
group of isometries, $G$, that respects the splitting. Then either 
\begin{enumerate}
\item $M=B\times F$ and $u$ is constant, or 
\item $M$ is a one-dimensional extension such that 
\begin{align*}
g_{M}=dr^{2}+g_{r}+u^2g_{\mathbb{R}^{k}}\quad\text{and}\quad u=be^{ar}.
\end{align*}
\end{enumerate}
\end{lemma}

\begin{proof} Since $G$ splits we have the projection $\pi:G\to\mathrm{Isom}(B)$
given by $\pi(\gamma)=\gamma_{1}$. Since $G$ acts transitively on
$M$, the image $\pi(G)$ acts transitively on $B$. By Proposition
\ref{Prop:IsomWP}, for all $\gamma_{1}\in\pi(G)$ there is a $C$
such that $\gamma_{1}^{*}(u)=Cu$, so by Lemma \ref{Lem:ssplit-rigid}
case (2) either $u$ is constant or $B$ is a one-dimensional extension,
$g_{B}=dr^{2}+g_{r}$ and $u=be^{ar}$. \end{proof}

\begin{theorem} \label{Thm:W} Let $(M,g)$ be a $G$-homogeneous
manifold and let $q$ be a $G$-invariant two-tensor. If $W$ is nontrivial,
then $(M^{n},g)$ is isometric to one of the following 
\begin{enumerate}
\item a space of constant curvature with $\dim W=n+1$, 
\item the product of a homogeneous space and a space of constant curvature
with $W$ consisting of functions on the constant curvature factor
with $2\leq\dim W\leq n$
\item the quotient of the product of a homogeneous space and $\mathbb{R}$,  $(H\times \mathbb{R})/\pi_1(M)$, 
with $W=\{w:\mathbb{R}\rightarrow \mathbb{R} \mid w''=\tau w\}$ where $\tau<0$ is constant, or 
\item a one-dimensional extension with $\dim W=1$. 
\end{enumerate}
\end{theorem}

\begin{proof} If $\mathrm{dim}(W)=1$, then we obtain a one-dimensional
extension by Theorem \ref{Thm:Wdim1}. Assume $M$ is not a space of constant curvature.  Then, if $\mathrm{dim}(W)>2$ or if $M$ is simply connected and $\mathrm{dim}(W)=2$, then from Theorem \ref{Thm:WPR} we obtain the warped
product splitting $M=B\times_{u}F$ and we have that all $w$ are
of the form $w(x,y)=u(x)v(y)$. First we want to show that $u>0$.
To see this suppose that $u(x_{0})=0$ for some $x_{0}$, then $w(x_{0},y)=u(x_{0})v(y)=0$,
so there is a singular point where all functions in $w$ vanish. But
since $G$ acts on $W$ and is transitive this would imply that all
functions in $W$ are zero, a contradiction.

Next we observe that $G$ respects the splitting $M=B\times_{u}F$.
In fact, the tangent distributions to the leaves $\{b\}\times F$
are given by $\mathcal{F}=\{\nabla w\mid w\in W_{p}\}$ where $W_{p}=\{w\in W\mid w(p)=0\}$.
Since $G$ preserves $W$ it must also preserve $\mathcal{F}$ as
well as the orthogonal distribution.

In case $M$ is not simply connected and $\mathrm{dim}(W)=2$ we reach the 
same conclusion for the universal cover of $M$. Here $W=\{w:\mathbb{R}\rightarrow \mathbb{R} \mid w''=\tau w\}$ becomes a space of functions on $\mathbb{R}$ that is invariant under a cyclic group of translations. 
Since our quadratic form is invariant under a homogeneous group the function $\tau$ must be constant.

We can now apply Lemma \ref{Lemma:WPHomogeneous} to see that either
$M$ is a one-dimensional extension, a direct product, or the universal cover is a direct product with $\mathbb{R}$. Once $M$ or its universal cover
is a direct product we have that $u$ is constant, so $w=u(x)v(y)$
shows that all the functions in $W$ are only on the constant curvature
factor, $F$.
 \end{proof}
Now we consider what more we can say in the case that $q$ is assumed to be divergence free or Codazzi.  Note that is cases (1)-(3) of Theorem \ref{Thm:W} $q$ is either a constant multiple of the metric or, on the products,  a constant sum of the metrics on the factors.  In particular, $q$ is both divergence free and Codazzi.  We show that the Codazzi property in fact characterizes these examples, while there are many more examples which are divergence free.  First we establish some properties of the metrics in case (4) of the previous theorem. 
\begin{proposition}Let $w=e^{ar},\,a>0$, where $r:M\rightarrow\mathbb{R}$
is a distance function. If $q=\frac{1}{w}\hess w$ , then 
\begin{align}
\label{5.1}q&=a^{2}dr^{2}+a\hess r,\\
\label{5.2} \left(\nabla_{X}Q\right)\left(\nabla w\right) &=  a^{2}wQ\left(X\right)-wQ^{2}\left(X\right)
\end{align}
where $Q$ dual $(1,1)$ tensor to $q$.  If $q$ is divergence free we further have: 
\begin{align}
\label{5.3} \tr q^{2} & =  \tr q,\\
\label{5.4} \left|\hess r\right|^{2} & =  a\Delta r.
\end{align}
In particular, if $q$ is invariant under a transitive group of isometries,
then so is $\hess r$.

\end{proposition}

\begin{proof} (\ref{5.1}) follows directly from $q=\frac{1}{w}\hess w$
as $w=e^{ar}$. To prove (\ref{5.2})  note that we have that $wQ\left(X\right)=\nabla_{X}\nabla w$.
so we obtain 
\begin{eqnarray*}
\left(\nabla_{X}Q\right)\left(\nabla w\right) & = & \nabla_{X}Q\left(\nabla w\right)-Q\left(\nabla_{X}\nabla w\right)\\
 & = & a^{2}\nabla_{X}\nabla w-wQ^{2}\left(X\right)\\
 & = & a^{2}wQ\left(X\right)-wQ^{2}\left(X\right),
\end{eqnarray*}
where the formula $Q\left(\nabla w\right)=a^{2}\nabla w$ follows
from (\ref{5.1}). 

Tracing (\ref{5.1}) also gives us 
\[
\tr q=a^{2}+a\Delta r
\]
and 
\[
\tr q^{2}=\left|q\right|^{2}=a^{4}+a^{2}\left|\hess\right|^{2}.
\]
Thus (\ref{5.4}) follows from (\ref{5.3}).  To see (\ref{5.3}), consider the trace of (\ref{5.2}) 
\[
\mathrm{div}q\left(\nabla w\right)=w\left(a^{2}\tr q-\tr q^{2}\right),
\]
which implies (\ref{5.3}).

\end{proof}

We now show the characterization in the Codazzi case. 

\begin{theorem}
With $(M,g)$ and $q$ as in Theorem \ref{Thm:W},  $q$ is Codazzi if and only if $(M,g)$ is isometric to one of the cases (1)-(3). 
\end{theorem}

\begin{proof}
The fact that $q$ is Codazzi, (\ref{5.2}) and  $\frac{\nabla w}{w} = a \nabla r$, implies that 
\begin{align}  \label{5.5} a(\nabla_{\nabla r} q)(X,X) = a^2 q(X,X) - q^2(X,X). \end{align}
At a point $p$, let $X$ be an eigenvector for $q$ perpendicular to $\nabla r$, with eigenvalue $\lambda$.  Let $\beta$ be the geodesic at $p$  in the $\nabla r$ direction and let $\phi_t$ be a smooth curve of isometries in $G$ such that $\phi_t(p) = \beta(t)$.  Define $X_t = d\phi_t(X)$.  Then $X_t$ is a vector field along $\beta$ with $|X_t| =1$. Since $\phi_t$ preserves $\nabla r$ and $q$ is invariant under $\phi_t$ we also have that $X_t \perp \nabla r$ and $X_t$ an eigenvector of $q$ with eigenvalue $\lambda$ for all $t$.  Using $X_t$ we can then calculate, 
\begin{align*}
(\nabla_{\nabla r} q)(X,X) &= D_{\nabla r} (\lambda) - 2 q(\nabla_{\nabla r} X, X) \\
&= D_{\nabla r} (\lambda) -2 \lambda D_{\nabla r} |X_t|^2 \\
&=0. 
\end{align*}
Plugging this back into (\ref{5.5}) gives that either $\lambda = 0$ or $\lambda = a^2$ so  $q$ has only two possible eigenvalues.  By invariance of $q$,  the multiplicity of the eigenvalues is constant, so the corresponding eigenspace decomposition gives us a pair of orthogonal distributions on $M$.  Moreover, since $q$ is Codazzi, these eigendistributions are integrable (See Chapter 16 of \cite{Besse}). Consequently, we can write the one dimensional extension, $M$, as 
\begin{align*}
M & = \mathbb{R} \times N_1 \times N_2 \\
g &= dr^2 + (g^1)_r + (g^2)_r 
\end{align*}
where the tangent space to $N_1$ corresponds to the eigenvectors for $q$ with eigenvalue $a^2$ and the tangent space to $N_2$ corresponds to nullvectors for $q$.  But then (\ref{5.1}) implies that 
\begin{align*}
\mathrm{Hess}r = a (g^1)_r
\end{align*}
Which implies that we have a warped product splitting
\begin{align*}
g = dr^2 + e^{2ar} (g^1)_0 + g^2_0.
\end{align*} Since the group $G$ preserves the eigenspaces of $q$ it preserves the warped product splitting in the sense of Lemma \ref{Lemma:WPHomogeneous}  and then the Lemma implies that $(g^1)_0$ is a flat metric on Euclidean space. Then we have that $ dr^2 + e^{2ar} (g^1)_0$ is a hyperbolic metric. 

Putting this all together we have three cases, if the only eigenvalue if $q$ is $a^2$ then $M$ is hyperbolic space which is contained in (1) of Theorem \ref{Thm:W}, if the only eigenvalue of $q$ is $0$ then we have a direct product as in case (3) of Theroem \ref{Thm:W}, finally if both eigevnvalues occur we have a product of a homogeneous space and hyperbolic space as in case (2). 
\end{proof}

Now we consider the divergence free case.  The only case we need
to consider is evidently when $\dim W=1$ and is spanned by $w=e^{ar},\,a>0$,
where $r:M\rightarrow\mathbb{R}$ is a distance function. An interesting
special case occurs when 
\[
\hess w=\frac{w}{m}\left(\ric-\lambda g\right).
\]
This is the so-called quasi-Einstein or warped product Einstein equation
as it is the equation on $B$ that makes a warped product $B\times_{w}F$
an Einstein metric when $F$ is an appropriately chosen Einstein metric. Interestingly
there are many such examples that are 1-dimensional extensions of
algebraic solitons (see \cite{He-Petersen-Wylie-Homogeneous}, \cite{Lafuente}).
The quasi-Einstein equation is  studied in more detail in section 7.

With these examples in mind we cannot expect the same rigid behavior in the divergence free case.
In fact, we will produce examples of one-dimensional extensions $G=H\rtimes\mathbb{R}$
such that $H$ is not an algebraic soliton and $\frac{1}{w}\hess w$
is divergence free, where $w=e^{ar}$.

Before discussing the examples, we identify some situations where we do obtain products
and warped products.

\begin{corollary}\label{cor:div_rigid}Let $w=e^{ar},\,a>0$, where
$r:M\rightarrow\mathbb{R}$ is a distance function on a homogeneous
space $\left(M,g\right)$. If $q=\frac{1}{w}\hess w$ is invariant
under a transitive group of isometries and divergence free, then $\Delta r\in\left[0,\left(n-1\right)a\right]$.
When $\Delta r=0$, the metric splits as a product $g=dr^{2}+g_{0}$,
and when $\Delta r=\left(n-1\right)a\neq0$ the metric is isometric
to $H^{n}\left(-a^{2}\right)$. Moreover, these are the only possibilities
for $g$ to be a warped product of the type $dr^{2}+\rho^{2}\left(r\right)g_{N}$,
where $\rho:\bbR\rightarrow\left(0,\infty\right)$.

\end{corollary}

\begin{proof}From the last formula in the previous proposition and
Cauchy-Schwarz we have 
\[
\frac{(\Delta r)^{2}}{n-1}\leq|\hess r|^{2}=a\Delta r.
\]
This establishes the range of possible values for $\Delta r$. When
$\Delta r=0$, the Hessian vanishes and we obtain a product metric.
While when $\Delta r$ is maximal we must have that $\hess r=ag_{r}$,
where $g=dr^{2}+g_{r}$. This shows that $L_{\nabla r}g_{r}=2ag_{r}$
and consequently that $g_{r}=e^{2ar}g_{0}$. This shows that 
\[
q=a^{2}dr^{2}+a\hess r=a^{2}g.
\]
When $a\neq0$ this shows that $\nabla e^{ar}$ is a conformal field
that is not a Killing field. As the metric is homogeneous we conclude
that it must be locally conformally flat. Theorem \ref{thm:Takagi}
then shows that the space is isometric to $H^{n}\left(-a^{2}\right)$.

Finally if we assume that $g=dr^{2}+\rho^{2}\left(r\right)g_{N}$,
then $\hess r=\frac{\rho'}{\rho}g_{r}$. In particular, the Cauchy-Schwarz
inequality $\frac{(\Delta r)^{2}}{n-1}\leq|\hess r|^{2}$ must be
an equality. This forces us to be in one of the two previous situations.

\end{proof}

The goal for the remainder of the section is to construct examples
indicating that there is little hope for classifying the general situation
where $q$ is divergence free. To that end it is convenient to use the following
condition.

\begin{proposition} Assume that $G$ is a transitive group of isometries
on $M$ and that $r:M\rightarrow\mathbb{R}$ is a smooth distance
function whose Hessian is invariant under $G$. If the hypersurface
$N=\left\{ x\in M\mid r\left(x\right)=0\right\} $, has divergence
free second fundamental form at one point, then it is possible to
find $a\in\mathbb{R}$ such that $q=\frac{1}{w}\hess w$ is divergence
free and $G$ invariant, where $w=e^{ar}$.

\end{proposition}

\begin{proof}First note that as $G$ acts transitively we only need
to check that an invariant tensor is divergence free at a specific point $p$.

When $w=e^{ar}$ we have that $q=\frac{1}{w}\hess w=a^{2}dr^{2}+a\hess r$.
Thus $q$ is also invariant under $G$. The divergence is: 
\[
\mathrm{div}q=\mathrm{div}\hess r+a\Delta rdr.
\]
By invariance, it follows that $\mathrm{div}\hess r(\nabla r)$ is
constant. Thus we can choose $a$ so that $\mathrm{div}q\left(\nabla r\right)=0$.
This shows that we obtain $\mathrm{div}q=0$ when $\mathrm{div}\hess r(X)=0$
for $X\perp\nabla r$. As $\hess r$ is the second fundamental form
for the level sets for $r$ we need to check that $\mathrm{div}\hess r(X)=\mathrm{div}_{N}\mathrm{II}\left(X\right)$.
This follows provided $\left(\nabla_{\nabla r}\hess r\right)\left(\nabla r,X\right)=0$
and that calculating this divergence intrinsically on $N$ is the
same as calculating it with the connection on $M$. We will check
this for the type changed $\left(1,1\right)$-tensor $S\left(X\right)=\nabla_{X}\nabla r$.
For the intrinsic part use an orthonormal frame $E_{i}$ for $N$
: 
\begin{eqnarray*}
\left(\nabla_{E_{i}}^{M}\mathrm{II}\right)\left(E_{i},X\right) & = & g\left(\left(\nabla_{E_{i}}^{M}S\right)\left(X\right),E_{i}\right)\\
 & = & g(\nabla_{E_{i}}(S(X))-S(\nabla_{E_{i}}X),E_{i})\\
 & = & g(\nabla_{E_{i}}^{N}(S(X))+g(\nabla_{E_{i}}(S(X)),\nabla r)\nabla r-S(\nabla_{E_{i}}^{N}X)-g(\nabla_{E_{i}}X,\nabla r)S(\nabla r),E_{i})\\
 & = & g\left(\left(\nabla_{E_{i}}^{N}S\right)\left(X\right),E_{i}\right),
\end{eqnarray*}
since $\nabla r\perp E_{i}$ and $S\left(\nabla r\right)=0$. Finally,
we also have 
\begin{eqnarray*}
\left(\nabla_{\nabla r}\hess r\right)\left(\nabla r,X\right) & = & g\left(\left(\nabla_{\nabla r}S\right)\left(\nabla r\right),X\right)\\
 & = & g\left(\nabla_{\nabla r}\left(S\left(\nabla r\right)\right),X\right)-g\left(S\left(\nabla_{\nabla r}\nabla r\right),X\right)\\
 & = & 0.
\end{eqnarray*}

\end{proof}

The general set-up for constructing a 1-dimensional extension is a
Lie group $H$ with a derivation $D$ on the Lie algebra $\mathfrak{h}$.
This gives us a Lie algebra $\mathfrak{g}=\mathfrak{h}\rtimes\mathbb{R}$
and corresponding Lie group $G$. The metric is left invariant and
preserves orthogonality in the semi-direct splitting $T_{e}G=\mathfrak{g}=\mathfrak{h}\rtimes\mathbb{R}$.
Thus it is determined by a left invariant metric on $H$. Finally,
as in \cite{He-Petersen-Wylie-Homogeneous}, the tensor $T$ that
corresponds to the second fundamental form for $H$ is proportional
to the symmetric part of the derivation.

Specifically, fix an $n$-dimensional Lie group $H$ and a left invariant
basis $X_{i}$ for its Lie algebra $\mathfrak{h}$. The structure
constants are given by 
\[
\left[X_{i},X_{j}\right]=c_{ij}^{k}X_{k}.
\]
The Lie group is said to be unimodular if $\mathrm{tr}\left(\mathrm{ad}_{X}\right)=0$
for all $X$. This is equivalent to $c_{ij}^{j}=0$ for all $i$.
We fix a derivation $D$, but in what follows the derivation property
is not used, only that it is a linear operator on the Lie algebra.

Our calculations will be with respect to a general left invariant
metric $g_{ij}=g\left(X_{i},X_{j}\right)$. The corresponding connection
is given by 
\begin{eqnarray*}
2g\left(\nabla_{X_{i}}X_{j},X_{k}\right) & = & g\left(\left[X_{i},X_{j}\right],X_{k}\right)-g\left(\left[X_{i},X_{k}\right],X_{j}\right)-g\left(\left[X_{j},X_{k}\right],X_{i}\right)\\
 & = & g_{kl}c_{ij}^{l}-g_{jl}c_{ik}^{l}-g_{il}c_{jk}^{l}.
\end{eqnarray*}

The symmetric part of $D$ is given by 
\begin{eqnarray*}
S & = & \frac{1}{2}D+\frac{1}{2}D^{*},\\
S_{j}^{i} & = & \frac{1}{2}D_{j}^{i}+\frac{1}{2}g^{il}\left(D^{t}\right)_{l}^{k}g_{kj}=\frac{1}{2}D_{j}^{i}+\frac{1}{2}g^{il}D_{k}^{l}g_{kj}.
\end{eqnarray*}
This can be type changed to two symmetric bilinear forms: $S^{ij}$
and $S_{ij}$. Note that 
\[
S_{i}^{k}g_{kj}=S_{ij}=g\left(S\left(X_{i}\right),X_{j}\right)=S_{ji}=S_{j}^{k}g_{ki}
\]
and similarly 
\[
S_{k}^{i}g^{kj}=S^{ij}=S_{k}^{j}g^{ki}.
\]

\begin{proposition}

With these assumptions and notation it follows that:

\[
2g\left(\mathrm{div}S,X\right)=\mathrm{tr}\left(D\circ\mathrm{ad}_{X}\right)+g\left(D,\mathrm{ad}_{X}\right)-2\mathrm{tr}\left(\mathrm{ad}_{S\left(X\right)}\right).
\]

\end{proposition}

\begin{proof}

The goal is to calculate $\mathrm{div}S=g^{ij}\left(\nabla_{X_{i}}S\right)\left(X_{j}\right)$.
Since it is easier to calculate the corresponding 1-form we calculate
instead: 
\begin{eqnarray*}
2g^{ij}g\left(\left(\nabla_{X_{i}}S\right)\left(X_{j}\right),X_{k}\right) & = & 2g^{ij}g\left(\nabla_{X_{i}}S\left(X_{j}\right),X_{k}\right)-2g^{ij}g\left(\nabla_{X_{i}}X_{j},S\left(X_{k}\right)\right)\\
 & = & 2g^{ij}S_{j}^{\alpha}g\left(\nabla_{X_{i}}X_{\alpha},X_{k}\right)-2g^{ij}S_{k}^{\alpha}g\left(\nabla_{X_{i}}X_{j},X_{\alpha}\right)\\
 & = & g^{ij}S_{j}^{\alpha}\left(g_{k\beta}c_{i\alpha}^{\beta}-g_{i\beta}c_{\alpha k}^{\beta}-g_{\alpha\beta}c_{ik}^{\beta}\right)-g^{ij}S_{k}^{\alpha}\left(g_{\alpha\beta}c_{ij}^{\beta}-g_{i\beta}c_{j\alpha}^{\beta}-g_{j\beta}c_{i\alpha}^{\beta}\right)\\
 & = & S^{i\alpha}g_{k\beta}c_{i\alpha}^{\beta}-S_{j}^{\alpha}c_{\alpha k}^{j}-S^{i\alpha}g_{\alpha\beta}c_{ik}^{\beta}-g^{ij}S_{k}^{\alpha}g_{\alpha\beta}c_{ij}^{\beta}+S_{k}^{\alpha}c_{j\alpha}^{j}+S_{k}^{\alpha}c_{i\alpha}^{i}\\
 & = & S^{i\alpha}c_{i\alpha}^{\beta}g_{k\beta}-S_{j}^{\alpha}c_{\alpha k}^{j}-S_{\beta}^{i}c_{ik}^{\beta}-g^{ij}c_{ij}^{\beta}S_{k}^{\alpha}g_{\alpha\beta}+S_{k}^{\alpha}c_{j\alpha}^{j}+S_{k}^{\alpha}c_{i\alpha}^{i}\\
 & = & 0-2S_{j}^{\alpha}c_{\alpha k}^{j}-0+2S_{k}^{\alpha}c_{i\alpha}^{i}\\
 & = & 2\mathrm{tr}\left(S\circ\mathrm{ad}_{X_{k}}\right)-2\mathrm{tr}\left(\mathrm{ad}_{S\left(X_{k}\right)}\right)\\
 & = & \mathrm{tr}\left(D\circ\mathrm{ad}_{X_{k}}\right)+\mathrm{tr}\left(D^{*}\circ\mathrm{ad}_{X_{k}}\right)-2\mathrm{tr}\left(\mathrm{ad}_{S\left(X_{k}\right)}\right)\\
 & = & \mathrm{tr}\left(D\circ\mathrm{ad}_{X_{k}}\right)+g\left(D,\mathrm{ad}_{X_{k}}\right)-2\mathrm{tr}\left(\mathrm{ad}_{S\left(X_{k}\right)}\right).
\end{eqnarray*}
In other words: 
\[
2g\left(\mathrm{div}S,X\right)=\mathrm{tr}\left(D\circ\mathrm{ad}_{X}\right)+g\left(D,\mathrm{ad}_{X}\right)-2\mathrm{tr}\left(\mathrm{ad}_{S\left(X\right)}\right).
\]

\end{proof}

With a view toward concrete examples note that: $\mathrm{tr}\left(D\circ\mathrm{ad}_{X}\right)$
does not depend on the metric; while $\mathrm{tr}\left(\mathrm{ad}_{S\left(X\right)}\right)=0$
when the the Lie group is unimodular. Keep in mind that $g\left(D,\mathrm{ad}_{X}\right)$
is not linear in $g_{ij}$, in the given frame it looks like 
\[
g^{ij}g\left(D\left(X_{i}\right),\mathrm{ad}_{X}\left(X_{j}\right)\right)=g^{ij}g_{\alpha\beta}D_{i}^{\alpha}\left(\mathrm{ad}_{X}\right)_{j}^{\beta}.
\]

\begin{example}

The simplest examples are on the 3-dimensional Heisenberg group. This
algebra has the single relation: $\left[X,Y\right]=Z$. In this basis
the adjoint actions have the matrices 
\[
\mathrm{ad}_{X}=\left[\begin{array}{ccc}
0 & 0 & 0\\
0 & 0 & 0\\
0 & 1 & 0
\end{array}\right],\,\mathrm{ad}_{Y}=\left[\begin{array}{ccc}
0 & 0 & 0\\
0 & 0 & 0\\
-1 & 0 & 0
\end{array}\right],\,\mathrm{ad}_{Z}=\left[\begin{array}{ccc}
0 & 0 & 0\\
0 & 0 & 0\\
0 & 0 & 0
\end{array}\right]
\]
We use any derivation of the form: 
\[
D=\left[\begin{array}{ccc}
\lambda_{1} & 0 & 0\\
0 & \lambda_{2} & 0\\
0 & 0 & \lambda_{3}
\end{array}\right].
\]
The composition of this derivation with any of the adjoint actions
clearly vanishes. So for any metric we get the three equations: 
\begin{eqnarray*}
g\left(D,\mathrm{ad}_{X}\right) & = & \lambda_{1}g_{31}g^{21}+\lambda_{2}g_{32}g^{22}+\lambda_{3}g_{33}g^{23}=0,\\
g\left(D,\mathrm{ad}_{Y}\right) & = & -\lambda_{1}g_{31}g^{11}-\lambda_{2}g_{32}g^{12}-\lambda_{3}g_{33}g^{31}=0,\\
g\left(D,\mathrm{ad}_{Z}\right) & = & 0.
\end{eqnarray*}
These equations are clearly satisfied for any metric of the form 
\[
\left[\begin{array}{ccc}
g_{11} & g_{12} & 0\\
g_{12} & g_{22} & 0\\
0 & 0 & g_{33}
\end{array}\right]
\]
as the inverse satisfies $g^{23}=g^{13}=0$.

\end{example}

\begin{example}

Consider the three dimensional simple (and unimodular) Lie algebra
with relations: 
\[
\left[X_{1},X_{2}\right]=X_{3}, \left[X_{2},X_{3}\right]=X_{1}, \left[X_{3},X_{1}\right]=X_{2}.
\]
In this basis we have 
\[
\mathrm{ad}_{X_{1}}=\left[\begin{array}{ccc}
0 & 0 & 0\\
0 & 0 & -1\\
0 & 1 & 0
\end{array}\right],\,\mathrm{ad}_{X_{2}}=\left[\begin{array}{ccc}
0 & 0 & 1\\
0 & 0 & 0\\
-1 & 0 & 0
\end{array}\right],\,\mathrm{ad}_{X_{3}}=\left[\begin{array}{ccc}
0 & -1 & 0\\
1 & 0 & 0\\
0 & 0 & 0
\end{array}\right]
\]
We will use $D=\mathrm{ad}_{X_{1}}$. This derivation is skew-symmetric
with respect to the standard biinvariant metric that makes the basis
elements have equal length and be orthogonal. To calculate $\mathrm{tr}\left(D\circ\mathrm{ad}_{X}\right)$
we note that: 
\begin{eqnarray*}
\mathrm{tr}\left(D\circ\mathrm{ad}_{X_{1}}\right) & = & -2,\\
\mathrm{tr}\left(D\circ\mathrm{ad}_{X_{2}}\right) & = & 0,\\
\mathrm{tr}\left(D\circ\mathrm{ad}_{X_{3}}\right) & = & 0.
\end{eqnarray*}
Next we find $g\left(D,\mathrm{ad}_{X}\right)$ for a general metric:
\begin{eqnarray*}
g\left(D,\mathrm{ad}_{X_{1}}\right) & = & g_{33}g^{22}+g_{22}g^{33}-2g_{23}g^{23},\\
g\left(D,\mathrm{ad}_{X_{2}}\right) & = & -g_{33}g^{21}+g_{31}g^{23}-g_{21}g^{33}+g_{23}g^{31},\\
g\left(D,\mathrm{ad}_{X_{3}}\right) & = & g_{32}g^{21}-g_{13}g^{22}-g_{22}g^{31}+g_{21}g^{32}.
\end{eqnarray*}
We then restrict attention a metric of the form 
\[
\left[\begin{array}{ccc}
g_{11} & 0 & 0\\
0 & g_{22} & g_{23}\\
0 & g_{23} & g_{33}
\end{array}\right]
\]
The inverse is 
\[
\left[\begin{array}{ccc}
\frac{1}{g_{11}} & 0 & 0\\
0 & \frac{g_{33}}{g_{22}g_{33}-g_{23}^{2}} & -\frac{g_{23}}{g_{22}g_{33}-g_{23}^{2}}\\
0 & -\frac{g_{23}}{g_{22}g_{33}-g_{23}^{2}} & \frac{g_{22}}{g_{22}g_{33}-g_{23}^{2}}
\end{array}\right]
\]
and 
\begin{eqnarray*}
g\left(D,\mathrm{ad}_{X_{1}}\right) & = & 2\frac{g_{22}g_{33}}{g_{22}g_{33}-g_{23}^{2}}-2\frac{g_{23}^{2}}{g_{22}g_{33}-g_{23}^{2}}=2,\\
g\left(D,\mathrm{ad}_{X_{2}}\right) & = & -g_{33}\cdot0+0\cdot g^{23}-0\cdot g^{33}+g_{23}\cdot0=0,\\
g\left(D,\mathrm{ad}_{X_{3}}\right) & = & g_{32}\cdot0-0\cdot g^{22}-g_{22}\cdot0+0\cdot g^{32}=0.
\end{eqnarray*}
This results in a 4-dimensional family of metrics with the property
that $\mathrm{div}S=0$. This family includes the Berger spheres.

\end{example}


\section{ Traceless $W$}

Now we consider the vector space of functions $\mathring{W}(M,g,q)$ satisfying (\ref{eqn 1.2a}).  We have the
following definition.

\begin{definition} Let $(M,g)$ be a Riemannian manifold and $q$
a quadratic form. The space of functions $\mathring{W}(M,g,q)$ is
essential if $\mathring{W}(M,g,q)\neq W(M,g,q')$ for all quadratic
forms $q'$. \end{definition}

For simply connected spaces of constant curvature,  $\mathring{W}(M,g,0)$  is essential
since it is $\left(n+2\right)$-dimensional and $W(M,g,q)$ has maximal dimension $n+1$ \cite[Proposition 1.1]{He-Petersen-Wylie}. We have the following
result for essential/inessential $\mathring{W}$. The proof is exactly
analogous to the proposition in the $F$ case, so we omit it.

\begin{proposition} Let $(M,g)$ be a Riemannian manifold and $q$
a quadratic form. $\mathring{W}(M,g,q)$ is essential if and only
if $\mathring{W}(M,g,q)\neq W(M,g,q-\phi g)$ for all $\phi\in C^{\infty}(M)$.
Moreover, if $\mathring{W}(M,g,q)=W(M,g,q')$ is inessential and $q$
is invariant under $G\subset\mathrm{Isom}(M,g)$ then $q'$ is also
invariant under $G$. \end{proposition}

This gives the following characterization of essential $\mathring{W}$.

\begin{lemma} \label{Lemma:oWEssential} Let $(M,g)$ the a $G$-homogeneous
manifold and $q$ be a $G$-invariant tensor. If $\mathring{W}(M,g,q)$
is essential, then $(M,g)$ is locally conformally flat. \end{lemma}

\begin{proof} If $\mathrm{dim}(\mathring{W})=1$, take $w\in\mathring{W}$,
then clearly $\mathring{W}=W(M,g,q-\frac{\Delta w}{n}g)$, so $\mathring{W}$
is inessential. Therefore, $\mathrm{dim}(\mathring{W})>1$. Let $w_{1},w_{2}$
be linearly independent functions in $\mathring{W}$ and define $V=w_{1}\nabla w_{2}-w_{2}\nabla w_{1}$
and note that 
\begin{align*}
L_{V}g & =w_{1}\mathrm{Hess}w_{2}-w_{2}\mathrm{Hess}w_{1}=\frac{w_{1}\Delta w_{2}-w_{2}\Delta w_{1}}{n}g.
\end{align*}
So $V$ is a conformal field. If $V$ is Killing for all $w_{1},w_{2}\in\mathring{W}$,
then 
\begin{align}
w_{1}\mathrm{Hess}w_{2}=w_{2}\mathrm{Hess}w_{1}.\label{eqn:w12}
\end{align}
Let $p\in M$ and define $\mathring{W}_{p}=\{w\in\mathring{W}\mid w(p)=0\}$.
If $\mathring{W}_{p}=\mathring{W}$ at some point $p$, then all functions
in $\mathring{W}$ vanish at $p$. Since $\mathring{W}$ is invariant
under the transitive group of isometries, this would imply that $\mathring{W}$
is trivial. Therefore, $\mathring{W}_{p}\neq\mathring{W}$ $\forall p$.
In fact, if we define $q'$ by the formula 
\begin{align}
q'_{p}=\frac{\mathrm{Hess}_{p}w}{w(p)}\quad\text{where}\quad w\in W\setminus\mathring{W}_{p},\label{eqn:q'}
\end{align}
then $q'$ is well defined on all of $M$ by (\ref{eqn:w12}). We
then have that $\mathring{W}(q)=W(q')$ which contradicts that $W$
is essential.

Therefore, if $\mathring{W}$ is essential, then $(M,g)$ must support
a non-Killing conformal field and by Proposition \ref{Prop:CField}
the space is locally conformally flat. \end{proof}


In the $\mathring{F}$ case, we showed that any product space was
inessential, the following proposition shows that this is not the
case for $\mathring{W}$.

\begin{proposition} \label{Prop:oWProductCalc} Suppose that $(M^{n},g)=(M_1^{k}\times M_2^{n-k},g_{1}+g_{2})$
is a product manifold and $q=c_{1}g_{1}+c_{2}g_{2}$ where $c_{i}\in\mathbb{R}$
and $c_{1}\neq c_{2}$, then 
\[
\mathring{W}(M,g,q)=W(M_1,g_{1},\lambda g_{1})\oplus W(M_2,g_{2},-\lambda g_{2}),
\]
where $\lambda=c_{1}-c_{2}$. \end{proposition}

\begin{proof} If $w\in\mathring{W}$, then $\mathrm{Hess}w(X,U)=0$
for $X\in TM_1,\,U\in TM_2$, so by \cite[Proposition 2.1]{Petersen-Wylie-Symm},  $w=w_{1}+w_{2}$ where $w_i$ is a function on $M_i$. This
shows that for $X,Y\in TM_1$ 
\[
\mathring{\mathrm{Hess}w}(X,Y)=\mathrm{Hess}_{g_1}w_{1}(X,Y)-\frac{\Delta_{g_1}w_{1}+\Delta_{g_2}w_{2}}{n}g_{1}(X,Y).
\]
However, as 
\[
\mathring{q}=c_{1}g_{1}+c_{2}g_{2}-\frac{kc_{1}+\left(n-k\right)c_{2}}{n}\left(g_{1}+g_{2}\right)=\frac{n-k}{n}\lambda g_{1}-\frac{k}{n}\lambda g_{2}
\]
we also have 
\[
\mathring{\mathrm{Hess}w}(X,Y)=w\mathring{q}(X,Y)=\frac{n-k}{n}\lambda wg_{1}\left(X,Y\right).
\]
Setting these equations for $\mathring{\mathrm{Hess}w}$ equal  shows that
$\mathrm{Hess}_{g_1}w_{1}$ is conformal to $g_{1}$. Thus 
\[
\left(\frac{n-k}{n}\lambda w+\frac{\Delta_{g_1}w_{1}+\Delta_{g_2}w_{2}}{n}\right)g_{1}=\mathrm{Hess}_{g_1}w_{1}=\frac{\Delta_{g_1}w_{1}}{k}g_{1}.
\]
This implies that there is a constant $\alpha$ such that

\[
\frac{n-k}{nk}\Delta_{g_1}w_{1}-\frac{n-k}{n}\lambda w_{1}=\frac{n-k}{n}\lambda w_{2}+\frac{1}{n}\Delta_{g_2}w_{2}=\alpha.
\]
$\alpha$ is constant as the left side of the equation  depends $M_1$ only  and the right depends
on $M_2$ only. 

By assumption $\lambda\neq0$, so 
\begin{align*}
\mathrm{Hess}_{g_1}w_{1} & =\frac{\Delta_{g_1}w_{1}}{k}g_{1}=\left(\lambda w_{1}+\alpha\frac{n}{n-k}\right)g_{1}=\lambda\left(w_{1}+\frac{\alpha}{\lambda}\frac{n}{n-k}\right)g_{1},\\
\mathrm{Hess}_{g_2}w_{2} & =\frac{\Delta_{g_2}w_{2}}{n-k}g_{2}=\left(-\lambda w_{2}+\alpha\frac{n}{n-k}\right)g_{2}=-\lambda\left(w_{2}-\frac{\alpha}{\lambda}\frac{n}{n-k}\right)g_{2}.
\end{align*}

Taking 
\begin{align*}
w'_1 = w_{1}+\frac{\alpha}{\lambda}\frac{n}{n-k} \qquad 
w'_2 = w_{2}-\frac{\alpha}{\lambda}\frac{n}{n-k} 
\end{align*}
we then have $w=w_{1}+w_{2}=w'_{1}+w'_{2}$ where 
\begin{align*}
\mathrm{Hess}_{g_1}w'_{1} & =\lambda w'_{1},\\
\mathrm{Hess}_{g_2}w'_{2} & =-\lambda w'_{2}.
\end{align*}
\end{proof}

This gives us the following partial converse to Lemma \ref{Lemma:oWEssential}

\begin{proposition} Let $(M,g)$ be a simply connected homogeneous
locally conformally flat manifold. Then there is a unique $\mathrm{Isom}(M,g)$-invariant
trace free quadratic form $q$ such that $\mathring{W}(q)$ is an
essential $\left(n+2\right)$-dimensional space of functions. \end{proposition}

\begin{proof}   For a simply connected space form, since the isometry
group acts isotropically, the only $\mathrm{Isom}(M,g)$-invariant
trace free quadratic form is the zero tensor and we have already seen
that this is an essential $\left(n+2\right)$-dimensional space. For
the product cases in Theorem \ref{thm:Takagi}, since $q$ is assumed to be $\mathrm{Isom}(M,g)$-invariant
and the isometry groups split in these cases and act isotropically
on each factor, $q=c_{1}g_{B}+c_{2}g_{F}$. For $\bbR^{1}$, $\dim(W(\lambda g))=2$
for all $\lambda$. For $S^{k}(\kappa)$, $\dim(W(-\kappa g))=k+1$,
$\dim W(0)=1$, and $\dim(W(\lambda g))=0$ for $\lambda\neq0,-\kappa$.
For $H^{k}(-\kappa)$, $\dim(W(\kappa g))=k+1$, $\dim W(0)=1$, and
$\dim(W(\lambda g))=0$ for $\lambda\neq0,-\kappa$.

Proposition \ref{Prop:oWProductCalc} shows that if $c_{1}-c_{2}=-\kappa$
then $\mathring{W}$ has dimension $n+2$ and is essential. In addition
\[
\mathring{q}=c_{1}g_{B}+c_{2}g_{F}-\frac{kc_{1}+\left(n-k\right)c_{2}}{n}\left(g_{B}+g_{F}\right)=-\frac{n-k}{n}\kappa g_{B}+\frac{k}{n}\kappa g_{F}.
\]

\end{proof}

This now gives us the structure theorem for $\mathring{W}$.

\begin{theorem} \label{Thm:oW-structure} Let $(M,g)$ be a $G$-homogeneous
manifold and let $q$ be a $G$-invariant tensor such that there is
a non-constant function in $\mathring{W}(M,g,q)$, then $(M,g)$
is isometric to either 

\begin{enumerate}
\item a locally conformally flat space, 
\item a direct product of a homogeneous space and a space of constant curvature
with $\mathring{W}$ consisting of functions on the constant curvature
factor,
\item $(N\times \mathbb{R}/\pi_1(M)$ with  $W=\{w:\mathbb{R}\rightarrow \mathbb{R} \mid w''=\tau w\}$ where $\tau<0$ is constant, or 
\item a one-dimensional extension of a homogeneous space. 
\end{enumerate}
Moreover, when $(M,g)$ is not conformally flat, $\mathring{W}(q)=W(q')$,
where $q'$ is a $G$-invariant tensor of the form $q'=q-\lambda g$
for some $\lambda\in\mathbb{R}$. If, in addition, $q$ is Codazzi, then $(M,g)$ is isometric to one of the cases (1)-(3).  \end{theorem}

\begin{proof} If $\mathring{W}(q)$ is essential, then $(M,g)$ is
locally conformally flat. If $\mathring{W}(q)$ is inessential, then
$\mathring{W}(q)=W(q')$ and $q'=q-\phi g$. Since $q$ and $q'$
are both invariant under the transitive group $G$, $\phi$ is constant.
Then applying Theorem \ref{Thm:W} to $W(q')$ gives the result.

\end{proof}

Now we prove Theorem \ref{Thm:CompactIntro} from the introduction. 

\begin{theorem} Suppose that $(M,g)$ is a compact locally homogeneous
manifold and $q$ a local isometry invariant symmetric two tensor. 
\begin{enumerate}
\item If $\mathring{F}(q)$ contains a non-constant function, then $(M,g)$
is a spherical space form. 
\item If $\mathring{W}(q)$ is non-trivial, then $(M,g)$ is a direct product
of a homogeneous space $N$ and a spherical space form, isometric to $(N\times \mathbb{R})/\pi_1(M)$, or isometric
to $(S^{n-1}(\kappa)\times\bbR^{1})/\Gamma$. 
\end{enumerate}
In particular any positive function in  $\mathring{F}(q)$ or $\mathring{W}(q)$  must be constant.
\end{theorem}

\begin{proof}
Let $(M,g)$ is locally homogeneous and $f \in \mathring{F}(q)$, let $(\widetilde{M}, \widetilde{g})$ be the universal cover with covering metric and let $\widetilde{q} = \pi^* q$ be the pullback of $q$ to the universal cover.  Then the pullback function $\widetilde{f} = \pi^* f$ is in $\mathring{F}(\widetilde{M}, \widetilde{g}, \widetilde{q})$.   Then, since $q$ is invariant under local isometries of $(M,g)$, it is invariant under the isometry group of $(\widetilde{M}, \widetilde{g})$.  We can then apply  Theorem \ref{Thm:TF} to conclude that $(\widetilde{M}, \widetilde{g})$ is a sphere as $\widetilde{f}$ is a bounded function and none of the other possibilities given by Theorem \ref{Thm:TF} admit a function in $\mathring F$ which is bounded. 

The cases of $W$ and $\mathring W$ are similar in that we can apply Theorems \ref{Thm:W} and \ref{Thm:oW-structure} respectively to the universal covers and the only possibility for a having a bounded function are the ones given.

We note that none of these spaces in the conclusion of the theorem admit a positive function because all of the solutions on the sphere have zeroes.   
\end{proof}

\section{Quasi-Einstein and conformally Einstein metrics}
 \label{Section-QE}
In this section, we apply the structure theorems from the previous
sections for $W$ and $\mathring{W}$ to the tensor $q=\frac{1}{m}\mathrm{Ric}$
for a constant $m$. The corresponding equations 
\begin{align*}
\hess w & =\frac{w}{m}(\mathrm{Ric}-\lambda g)\\
\mathring{\hess w} & =\frac{w}{m}\mathring{\mathrm{Ric}}
\end{align*}
are often called the $m$-quasi Einstein and generalized $m$-quasi-Einstein
equations respectively. In the literature, the equations are often
considered in the case where $w$ is a positive function and then
the equations can be re-written in terms of $f$ when $w=e^{f}$,
but our results do not require $w$ to be positive. When $m>1$ is
an integer, solutions to the $m$-quasi Einstein equation correspond
to warped product Einstein metrics. Theorem \ref{Thm:GmQE} is a generalization
of structure results obtained by the authors with He in \cite{He-Petersen-Wylie-Homogeneous}.
When $m>1$, Lafuente in \cite{Lafuente} further showed that if $M$ is a homogeneous
$m$-quasi Einstein metric that is a one-dimensional extension of
a homogeneous space, $N$, then $N$ is an algebraic Ricci soliton.

Thus, directly combining this work with Theorem \ref{Thm:oW-structure},
we obtain a result for homogeneous generalized $m$-quasi Einstein
metrics when $m>1$.

\begin{theorem} \label{Thm:GmQE} Suppose $(M,g)$ is homogeneous
Riemannian manifold that admits a non-constant function $w$ which
solves the generalized $m$-quasi Einstein equation for some $m>1$.
Either $M$ is a locally conformally flat space, a product of an Einstein
metric and a space of constant curvature, the quotient of the product of a homogeneous space and $\mathbb{R}$,  $(H\times \mathbb{R})/\pi_1(M)$, or a one-dimensional extension
of an algebraic Ricci soliton metric. \end{theorem}

The construction in \cite{He-Petersen-Wylie-Homogeneous} also shows
that any algebraic Ricci soliton metric can be extended to a warped
product Einstein metric and that the derivation used to extend the
soliton is a multiple of the soliton derivation.

When $m<0$ the $m$-quasi Einstein equation does not seem to have
been studied in depth. In fact, we will see below that the question
of which spaces have one-dimensional extensions that are quasi-Einstein
is more complicated in this case. As a simple example of the difference
between the $m>0$ and $m<0$ cases consider the $m$-quasi Einstein
structures on $S^{n}$ and $H^{n}$.

\begin{example} Consider $S^{n}(\kappa)$ or $H^{n}(-\kappa)$, the
spaces of constant curvature $\pm\kappa$. Clearly $\mathrm{Ric}=\pm\kappa(n-1)g$,
but there are non-constant functions satisfying $\mathrm{Hess}w=\mp\kappa wg$.
So 
\[
\mathrm{Ric}-\frac{m}{w}\mathrm{Hess}w=\pm\kappa(n+m-1)g.
\]
In particular, when $m<-(n-1)$, then $H^{n}(-\kappa)$ has $\lambda>0$
and $S^{n}(-\kappa)$ has $\lambda<0$. Note that hyperbolic space
is a one-dimensional extension of Euclidean space, so it is possible
to have $\lambda>0$ for a one-dimensional extension, at least when
$m<-(n-1)$. \end{example}

Of special interest is the case $m=2-n$, $n\geq3$, where the equation
\[
\mathring{\mathrm{Hess}w}=\frac{w}{2-n}\mathring{\mathrm{Ric}}
\]
is the almost Einstein equation. If there is a positive solution to this equation we call the space conformally Einstein. Theorem \ref{Thm:oW-structure} shows that the only interesting homogeneous almost Einstein metrics are conformally Einstein. 

In dimension 4, homogeneous conformally
Einstein spaces are classified in \cite{4d-Conformal} by studying the Bach tensor of homogeneous $4$-manifolds. In
the classification, any non-symmetric space example is homothetic
to one of three families of one-dimensional extensions of $3$-dimensional
Lie algebras. One of the examples (case (ii) of \cite[Theorem 1.1]{4d-Conformal}) is a one-dimensional
extension of the Ricci soliton on the $3$-dimensional Heisenberg
group, the other two families are extensions of the abelian Lie algebra
and the extension derivations are not soliton derivations. In particular,
these non-soliton families have $\lambda=0$. Another difference when
$m<0$ is that not all algebraic solitons can be extended to $m$-quasi
Einstein metrics when $m<0$ as, for example, the solvable $3$-dimensional
soliton cannot be extended to a conformally Einstein metric.

Inspired by these examples we give two constructions of $m$-quasi
Einstein metrics for any dimension $n$ and parameter $m$. First we consider when we can extend an algebraic Ricci soliton to
an $m$-quasi Einstein metric for general $m$.

\begin{proposition} Let $(H^{n-1},h)$ be an algebraic Ricci soliton
metric 
\[
\mathrm{Ric}=\lambda I+D.
\]
There is a non-Einstein homogeneous $m$-quasi Einstein metric with Lie algebra
$\bbR\xi\ltimes\frak{h}$, where $\mathrm{ad}_{\xi}=\alpha D$ for
some constant $\alpha$, if and only if $\tr D>m\lambda$. \end{proposition}

\begin{remark} For an algebraic Ricci soliton, $\tr D>0$ and $\lambda<0$,
so the condition is trivially satisfied when $m>0$. Also note that
tracing the soliton equation gives $\tr(D)=\scal-(n-1)\lambda$, so
the condition is equivalent to $\scal>(n+m-1)\lambda$. For the conformal
Einstein case, $m=2-n$ the condition is $\scal>\lambda$. \end{remark}

\begin{remark} For the soliton on the three-dimensional Heisenberg group $\scal=\lambda/3$
while for the soliton on the three-dimensional Lie group Sol $\scal=\lambda$. In particular, the
three dimensional Heisenberg group can be extended to a conformally Einstein metric,
but Sol can only be extended to a $m$-quasi Einstein metric when
$m>-2$. \end{remark}

\begin{proof} By \cite[Lemma 2.9]{He-Petersen-Wylie-Homogeneous} the Ricci tensor of such a one-dimensional extension
is 
\begin{eqnarray}
\mathrm{Ric}\left(\xi,\xi\right) & = & -\alpha^{2}\tr(S^{2}),\nonumber \\
\mathrm{Ric}\left(X,\xi\right) & = & -\alpha\mathrm{div}(S),\label{eqn:RicS}\\
\mathrm{Ric}\left(X,X\right) & = & \mathrm{Ric}^{H}(X,X)-\left(\alpha^{2}\mathrm{tr}S\right)h\left(S(X),X\right)-\alpha^{2}h([S,A](X),X),\nonumber 
\end{eqnarray}
where $S=\frac{D+D^{t}}{2}$ and $A=\frac{D-D^{t}}{2}$. For an algebraic
Ricci soliton, $D$ is symmetric so $S=D$, $A=0$, $\mathrm{div}(D)=\mathrm{div}(\mathrm{Ric})=0$,
and $\mathrm{tr}(D^{2})=-\lambda\mathrm{tr}(D)$, so we have 
\begin{eqnarray*}
\mathrm{Ric}\left(\xi,\xi\right) & = & \lambda\alpha^{2}\tr D,\\
\mathrm{Ric}\left(X,\xi\right) & = & 0,\\
\mathrm{Ric}\left(X,X\right) & = & \lambda g+\left(1-\alpha^{2}\mathrm{tr}D\right)h\left(D(X),X\right).
\end{eqnarray*}
When we write $w=e^{ar}$, then $\mathrm{Hess}w=wa^{2}dr\otimes dr-wa\alpha h(S(\cdot),\cdot)$ (see the proof of \cite[Theorem 3.3] {He-Petersen-Wylie-Homogeneous})
and 
\begin{eqnarray*}
\left(\mathrm{Ric}-\frac{m}{w}\mathrm{Hess}w\right)(\xi,\xi) & = & \lambda\alpha^{2}\tr D-ma^{2},\\
\left(\mathrm{Ric}-\frac{m}{w}\mathrm{Hess}w\right)(X,X) & = & \lambda h(X,X)+(1-\alpha^{2}\mathrm{tr}D+ma\alpha)h\left(S(X),X\right).\\
\end{eqnarray*}
So, if we want to obtain $\mathrm{Ric}-\frac{m}{w}\mathrm{Hess}w=\lambda g$,
then we have to solve the equations

\begin{eqnarray*}
\lambda & = & \lambda\alpha^{2}\tr D-ma^{2},\\
1 & = & \alpha^{2}\tr D-ma\alpha
\end{eqnarray*}
for the unknown constants $\alpha$ and $a$. Multiplying the second
equation by $\lambda$ and subtracting the two equations gives that
either $a=0$ or $a=\alpha\lambda$. The $a=0$ case is the Einstein
case, so we take $a=\alpha\lambda$. Plugging this back into the system
gives 
\[
1=\alpha^{2}(\tr D-m\lambda)
\]
so there exists such an $\alpha$ if and only if $\tr D>m\lambda$.

\end{proof}

\begin{proposition} Let $\frak{h}$ be an abelian Lie algebra and
$D$ a normal derivation of $\frak{h}$ such that 
\[
\tr(S^{2})=-\frac{\tr(S)^{2}}{m},
\]
where $S=\frac{D+D^{t}}{2}$, then there is a homogeneous $m$-quasi
Einstein metric with Lie algebra $\bbR\xi\ltimes\frak{h}$ where $\mathrm{ad}_{\xi}=D$
and $\lambda=0$. \end{proposition}

\begin{remark} Taking $n=4$, $m=-2$, we obtain the condition that
$2\tr(S^{2})=\tr(S)^{2}$. The examples in \cite{4d-Conformal} have these
properties. \end{remark}

\begin{proof} We again use the equations (\ref{eqn:RicS}). Since
$\frak{h}$ is abelian, $\mathrm{Ric}^{H}=0$, and $\mathrm{div}(S)=0$
for any $D$, it follows that

\begin{eqnarray*}
\left(\mathrm{Ric}-\frac{m}{w}\mathrm{Hess}w\right)(\xi,\xi) & = & -\tr(S^{2})-ma^{2},\\
\left(\mathrm{Ric}-\frac{m}{w}\mathrm{Hess}w\right)(X,\xi) & = & 0,\\
\left(\mathrm{Ric}\left(X,X\right)-\frac{m}{w}\mathrm{Hess}w\right)(X,X) & = & -\left(\mathrm{tr}S-ma\right)h\left(S(X),X\right).
\end{eqnarray*}
When $a=\frac{\tr S}{m}$ the condition $\tr(S^{2})=-\frac{\tr(S)^{2}}{m}$
shows that both equations vanish. \end{proof}

Conversely, we have the following necessary conditions for any $m$-quasi
Einstein metric and, if the derivation is normal, the following
partial converse.


\begin{proposition} Suppose that there is a homogeneous $m$-quasi
Einstein metric with Lie algebra $\bbR\xi\ltimes\frak{h}$ where $\mathrm{ad}_{\xi}=D$
and $w=e^{ar}$. It follows that $\mathrm{div}(S)=0$ and $\tr(S^{2})=-a\tr(S)$.
Moreover, if $D$ is normal, then either $(H^{n-1},h)$ is a Ricci
soliton or $(H,h)$ is a flat space and $\tr(S^{2})=-\frac{(\tr(S))^{2}}{m}$.
\end{proposition}

\begin{proof} Consider again the equations (\ref{eqn:RicS}). First
note that $\left(\mathrm{Ric}-\frac{m}{w}\mathrm{Hess}w\right)\left(X,\xi\right)=0$
implies that $\mathrm{div}(S)=0$ is necessary.

We also have $q=\frac{\Hess w}{w}$, for $q$ with $\div q=0$. In
terms of $r$, this gives 
\[
\div\Hess r=-a\Delta rdr.
\]
By the Bochner identity, 
\[
\div\Hess r=\nabla\Delta r+\mathrm{Ric}(\xi)=\mathrm{Ric}(\xi).
\]
So using the equation $\mathrm{Ric}(\xi,\xi)=-\tr(S^{2})$ from (\ref{eqn:RicS})
we have 
\[
-\tr(S^{2})=\mathrm{Ric}(\xi,\xi)=\div\Hess r(\xi,\xi)=-a\Delta r=a\tr(S).
\]
Now, if $D$ is normal we obtain $[A,S]=0$ so the equations become

\begin{eqnarray*}
\left(\mathrm{Ric}-\frac{m}{w}\mathrm{Hess}w\right)(\xi,\xi) & = & -\tr(S^{2})-ma^{2},\\
\left(\mathrm{Ric}\left(X,X\right)-\frac{m}{w}\mathrm{Hess}w\right)(X,X) & = & \mathrm{Ric}^{H}(X,X)-\left(\mathrm{tr}S-ma\right)h\left(S(X),X\right).
\end{eqnarray*}
Let $\beta=\tr S-ma$. When $\beta\neq0$ we have $\mathrm{Ric}^{H}=\lambda+\beta S$,
so $H$ is a Ricci soliton.

Otherwise, for $\beta=0$ it follows that $\mathrm{Ric}^{H}=\lambda g$,
$\tr S=am$, and $\lambda=-\tr(S^{2})-ma^{2}.$ But then the equation
$\tr(S^{2})=-a\tr(S)$ implies that $\lambda=0$ and consequently $H$
is flat since  homogeneous Ricci flat metrics are flat \cite{AK}.

\end{proof}

We finish with a final characterization of spaces that are conformally
Einstein that comes from a different approach. 

\begin{lemma} Assume $\left(M^{n},g\right)$ has a one-dimensional
space of solutions to the conformal Einstein equation: 
\[
\mathring{\mathrm{Hess}w}=\frac{w}{2-n}\mathring{\mathrm{Ric}},
\]
i.e., $\tilde{g}=w^{-2}g$ is an Einstein metric. If $G$ is a transitive
group of isometries and $H\subset G$ is the co-dimension one normal
subgroup that fixes $w$, then $H$ acts isometrically on the conformally
changed Einstein metric $\tilde{g}$ and $G$ acts conformally. Moreover,
either 
\begin{enumerate}
\item $w$ is constant and $g$ is Einstein, 
\item $w=e^{ar}$ and $\left(M,g\right)$ is isometric to $H^{n}\left(-a^{2}\right)$,
or 
\item $w=e^{ar}$ and all conformal fields from the action of $G$ have
constant divergence with respect to $\tilde{g}$. 
\end{enumerate}
\end{lemma}

\begin{proof} Note that $G$ clearly acts conformally with respect
to $\tilde{g}$. If $G$ acts isometrically, then $w$ is forced to
be constant and vice versa. Thus we can assume that $w=e^{ar}$, $a>0$.
Since $H$ fixes $w$ it follows that it acts isometrically on $\tilde{g}$.
This shows that the Riemannian submersion $r:\left(M,g\right)\rightarrow\bbR$
can be altered to a Riemannian submersion $\frac{1}{aw}:\left(M,\tilde{g}\right)\rightarrow\left(0,\infty\right)$.
We let $\mathfrak{h}\subset\mathfrak{g}$ denote the Lie algebras
of vector fields on $M$ that correspond to $H\subset G$. On $\left(M,\tilde{g}\right)$
all of the fields in $\mathfrak{g}$ are conformal and the fields
in $\mathfrak{h}$ are Killing. Consider $Z\in\mathfrak{g}-\mathfrak{h}$
so that $L_{Z}\left(\tilde{g}\right)=\frac{2}{n}\left(\mathrm{div}_{\tilde{g}}Z\right)\tilde{g}$.
A well-known formula by Yano shows that if $u=\frac{\mathrm{div}_{\tilde{g}}Z}{n}$,
then 
\[
L_{Z}\tilde{\ric}=-\left(n-2\right)\hess_{\tilde{g}}u-\Delta u\tilde{g}.
\]
As $\tilde{\ric}$ is Einstein this implies that $u\in\mathring{V}\left(M,\tilde{g}\right)$.
If some nonzero $u$ is constant, then all fields in $\mathfrak{g}$
have constant divergence with respect to $\tilde{g}$ as in case (3).
Otherwise, we have a non-constant $u\in\mathring{V}\left(M,\tilde{g}\right)$.
This gives a local warped product structure for $\tilde{g}$. We claim
that it is global by showing that $u=u\left(r\right)$. Since $\mathfrak{h}\subset\mathfrak{g}$
is an ideal we have that $\left[X,Z\right]\in\mathfrak{h}$ for all
$X\in\mathfrak{h}$. Thus 
\[
0=L_{\left[X,Z\right]}\tilde{g}=L_{X}L_{Z}\tilde{g}-L_{Z}L_{X}\tilde{g}=L_{X}\left(2u\tilde{g}\right)=2\left(D_{X}u\right)\tilde{g}.
\]
This shows that $u$ is invariant under $H$ and hence that $u=u\left(r\right)$.
Thus 
\begin{eqnarray*}
\tilde{g} & = & w^{-2}g=dt^{2}+\varphi^{2}\left(t\right)g_{N},\\
g & = & w^{2}\left(dt^{2}+\rho^{2}g_{N}\right)=dr^{2}+\rho^{2}g_{N}.
\end{eqnarray*}
When the metric is inessential we can use Corollary \ref{cor:div_rigid}
to conclude that we are in case (2). In case it is essential we can
instead use Takagi's classification (see Theorem \ref{thm:Takagi})
to see that only hyperbolic space can admit solutions of the from
$w=e^{ar},a\neq0$, to the conformal Einstein equation.

\end{proof}

\appendix

\section{K\"{a}hler manifolds}

In this appendix we include a discussion of some of the spaces of functions discussed above on K\"{a}hler manifolds.  No isometric symmetry is assumed in this section, but we will assume that the tensor $q$ is Hermitian.

Recall that a K\"ahler manifold is a complex manifold, $M$, equipped with a Riemannian metric, $g$, such that the complex structure $J$ is skew-adjoint and parallel with respect to $g$. A symmetric $2$-tensor, $q$ is called \emph{Hermitian} if $q(Jv, w) = -q(v, Jw)$.  If $q$ is Hermitian, then $\chi(v,w) = q(Jv,w)$ defines a $2$-form.   Note that the Ricci tensor and metric of a K\"ahler manifold are Hermitian with closed $2$-form.  Thus, for K\"ahler gradient Ricci solitons, quasi Einstein metrics, and conformally Einstein metrics the tensor $q$ is Hermitian and the corresponding $2$-form is closed.  

In fact, the problem of when a K\"ahler manifold admits a non-trivial function with Hermitian Hessian has been investigated extensively by Derdzinski and Maschler where they obtain interesting results for  K\"ahler conformally Einstein manfiolds \cite{DM1, DM2, DM3}.   Note that functions with Hermitian Hessian are also called Killing potentials because Derdzinski and Maschler show that a function has Hermitian Hessian if and only if $J$ applied to the gradient is a Killing field.  Case, Shu, and Wei also obtain a rigidity result for K\"ahler quasi-Einstein metrics which says that they must be a quotient of a  product of a surface and an Einstein metric \cite[Theorem 1.3]{Case-Shu-Wei}.  In this appendix we verify that this result holds in general for functions in a solution space of the form $W(q)$ when $q$ is Hermitian and the corresponding $2$-form is closed. 

\begin{proposition} \label{Thm:Kaehler} Let $(M,g)$ be a simply connected  K\"ahler
manifold and  $q$ a Hermitian symmetric two-tensor such that the corresponding $2$-form is closed.  If $W(q)$ is nontrivial, then $(M,g)$ is an isometric product $N_1^2 \times N_2^{n-2}$ and  $W$ consists of  functions on the $N_1$ factor only. 
\end{proposition}

\begin{proof}
Let $w$ be a non-constant function such that $\mathrm{Hess}w = w q$. The proof proceeds as in \cite{Case-Shu-Wei} as the only properties used in the proof come from the general properties of $q$.  We include an outline of the proof for completeness. 

Let $\chi(v,w) = q(Jv,w)$, $\omega = g(Jv,w)$ and $\phi = \mathrm{Hess}w(Jv,w)= \frac{1}{2} L_{\nabla w} \omega$.  Then $\phi$ is closed as the Lie derivative of a closed form $\omega$.  By assumption $\frac{\phi}{w}$ is also closed as it is equal to $\chi$.  Therefore, $dw \wedge \phi = 0$. 

Then 
\begin{align*}
(dw \wedge \phi)(X,Y,Z) &= (D_X w) \phi(Y, Z) + (D_Y w) \phi(Z,X) + (D_Z w) \phi(X,Y) \\
&= (D_X w) g(\nabla_{JY} \nabla w, Z) + (D_Y w) g(\nabla_{JZ} \nabla w, X) + (D_Z w) g(\nabla_{JX} \nabla w, Y)
\end{align*}
Taking  $X,Y \perp \nabla w$, $Z =\nabla w$  then gives 
\[ 0 = |\nabla w|^2 g(\nabla_{JX} \nabla w, Y)= -|\nabla w|^2 g(\nabla_{X} \nabla w, JY).\]
So that $\nabla _X \nabla w \perp JY$ whenever $\nabla w  \neq 0$.  On the other hand, taking $X = \nabla w$, $Y = J\nabla w$ and $Z \perp \nabla w$ we also obtain
\[ 0 = |\nabla w|^2 \phi(JX,Z) = -  |\nabla w|^2 g(\nabla_{\nabla  w} \nabla w, Z).\]
Which implies that $\nabla_{\nabla w} \nabla w$ is parallel to $\nabla w$ when $\nabla w \neq 0$.

Putting this together shows that $\nabla_{\cdot}\nabla w \in \mathrm{span}\{ \nabla w, J\nabla w\}.$  The fact that $\mathrm{Hess} w$ is Hermitian also implies that  $\nabla_{\cdot}J \nabla w = J(\nabla_{\cdot}\nabla w).$ So we also have that $\nabla_{\cdot}J\nabla w \in \mathrm{span}\{ \nabla w, J\nabla w\}.$

This implies that $\mathrm{span}\{ \nabla w, J\nabla w\}$ is a parallel distribution on the set where $\nabla w  \neq 0$ and thus gives an isometric splitting on this set.  Since $\nabla_{\cdot} \nabla w = w q$ and $q$ is assumed to be smooth, we also have that this distribution is locally uniformly continuous, so that the isometric splitting extends to the closure of $\{ \nabla w \neq 0 \}$. However, since $w$ is a Killing potential,  by remark 5.4 in \cite{DM1},  $\nabla w \neq 0$ almost everywhere, so we have the isometric splitting on all of $M$. 
\end{proof}

In contrast to this result for $W(q)$ note that there are many interesting examples of K\"ahler Ricci solitons and K\"ahler conformally Einstein spaces, so no such strong rigidity is possible for the spaces of function $F(q)$, $\mathring{F}(q)$ or $\mathring{W}(q)$ when $q$ is assumed to be Hermitian.  See \cite{DM1, DM2, DM3} for further results on K\"ahler Killing potentials.

\end{document}